\documentclass[reqno]{amsart}
\usepackage{amssymb,color,epsf}
\usepackage [cmtip,arrow]{xy}
\xyoption{all}
\usepackage {pb-diagram,pb-xy}
\usepackage[mathscr]{eucal}
\usepackage [autostyle, english = american]{csquotes}
\MakeOuterQuote{"}

\ifx\pdfpageheight\undefined
   \usepackage[dvips,colorlinks=true,linkcolor=blue,citecolor=red,%
      urlcolor=green]{hyperref}
   \usepackage[dvips]{graphicx}
   \makeatletter
   \edef\Gin@extensions{\Gin@extensions,.mps}
   \makeatother
\else
 \usepackage[pdftex]{graphicx}
   \usepackage[bookmarksopen=false,pdftex=true,breaklinks=true,%
      backref=page,pagebackref=true,plainpages=false,%
      hyperindex=true,pdfstartview=FitH,colorlinks=true,%
      pdfpagelabels=true,colorlinks=true,linkcolor=blue,%
      citecolor=red,urlcolor=green,hypertexnames=false%
      ]%
   {hyperref}
\fi
\usepackage{bm}

\usepackage{amsmath,amssymb,amsfonts}
\newtheorem{theorem}{Theorem}[section]
\newtheorem{lemma}[theorem]{Lemma}

\newtheorem{proposition}[theorem]{Proposition}
\newtheorem{conjecture}[theorem]{Conjecture}

\theoremstyle{definition}
\newtheorem{definition}[theorem]{Definition}

\newtheorem{remark}[theorem]{Remark}

\definecolor{DarkBlue}{rgb}{0,0.1,0.55}

\numberwithin{equation}{section}

\newcommand {\hide}[1]{}

\newcommand {\junk}[1]{}

\newcommand {\R} {\mathbb{R}}

\newcommand {\C}     {\mathbb{C}}

\newcommand {\Z}  {\mathbb{Z}}
 \newcommand {\N}         {\mathbb{N}}
\newcommand {\Q}         {\mathbb{Q}}

\newcommand {\eps} {{\varepsilon}}

\newcommand {\PP}     {\mathbb{P}} 

\newcommand{\Gr}{\mathrm{Gr}}

\def\addots{\mathinner{\mkern1mu
\raise1pt\vbox{\kern7pt\hbox{.}}
\mkern2mu\raise4pt\hbox{.}\mkern2mu
\raise7pt\hbox{.}\mkern1mu}}

\newcommand{\Cc}{\operatorname{cc}}

\def\ov#1{{\overline{#1}}}

\def\wt#1{{\widetilde{#1}}}

\newcommand{\Hf}{{\operatorname{H}}}
\newcommand{\Hp}{{\operatorname{P}}}
\newcommand{\aff}{{\operatorname{aff}}}
\newcommand{\card}{\operatorname{card}}

\newcommand{\cc}{\operatorname{cc}}
\newcommand{\irr}{\operatorname{irr}}

\newcommand{\bad}{{\rm B}}
\newcommand{\good}{{\rm G}}
\newcommand{\betti}{\operatorname{b_{0}}}
\newcommand{\ball}{\operatorname{B}}

\newcommand{\lin}{\operatorname{lin}}

\newcommand{\cC}{{\mathcal C}}

\newcommand{\cG}{{\mathcal G}}
\newcommand{\cH}{{\mathcal H}}
\newcommand{\cI}{{\mathcal I}}

\newcommand{\cL}{{\mathcal L}}

\newcommand{\cP}{{\mathcal P}}
\newcommand{\cQ}{{\mathcal Q}}
\newcommand{\cR}{{\mathcal R}}
\newcommand{\cS}{{\mathcal S}}

\newcommand{\cV}{{\mathcal V}}

\newcommand{\bfa}{{\boldsymbol{a}}}
\newcommand{\bfb}{{\boldsymbol{b}}}

\newcommand{\bfp}{{\boldsymbol{p}}}

\newcommand{\bfw}{{\boldsymbol{w}}}
\newcommand{\bfx}{{\boldsymbol{x}}}

\newcommand{\bfz}{{\boldsymbol{z}}}

\begin{document}
\title[Partitioning on varieties and point-hypersurface
incidences]{Polynomial partitioning on varieties of codimension two and
  point-hypersurface incidences in four dimensions}

\author[Basu]{Saugata Basu}
\address{Department of Mathematics,
Purdue University. West Lafayette, IN 47906, U.S.A.}
\email{sbasu@math.purdue.edu}
\urladdr{\url{http://www.math.purdue.edu/~sbasu}}

\author[Sombra]{Mart\'\i n Sombra}
\address{ICREA \& Departament d'{\`A}lgebra i Geometria, Universitat
  de Barcelona.  Gran Via 585, 08007 Bar\-ce\-lo\-na, Spain}
\email{sombra@ub.edu}
\urladdr{\url{http://atlas.mat.ub.es/personals/sombra}}

\date{\today} 
\thanks{Basu and Sombra were partially supported by the IPAM research
  program ``Algebraic Techniques for Combinatorial and Computational
  Geometry''. Basu was also  partially supported by NSF grants
  CCF-0915954,  CCF-1319080 and DMS-1161629.  Sombra was also partially supported by the MINECO
  research project MTM2012-38122-C03-02. 
  }

\begin{abstract}
  We present a polynomial partitioning theorem for finite sets of
  points in the real locus of 
  an irreducible
  complex algebraic variety of
  codimension at most two. This result
  generalizes the polynomial partitioning theorem on the Euclidean
  space of Guth and Katz, and its extension to hypersurfaces by Zahl
  and by Kaplan, Matou{\v{s}}ek, Sharir and Safernov\'a.

  We also present a bound for the number of incidences between points
  and hypersurfaces in the four-dimensional Euclidean space. It is an
  application of our partitioning theorem together with the refined
  bounds for the number of connected components of a semi-algebraic
  set by Barone and Basu.
\end{abstract}

\maketitle

\vspace{-8mm}
\setcounter{tocdepth}{1}
\tableofcontents
\vspace{-8mm}

\section[Introduction]{Introduction}

The polynomial partitioning method was introduced by Guth and Katz in
their seminal paper \cite{GK:eddpp}. Applying it in conjunction with
the Elekes' framework~\cite{Elekes-Sharir11}, they made a breakthrough
in a long-standing problem of Erd{\H o}s on the number of distinct
distances between points in the plane, by nearly proving the distinct
distances conjecture. Subsequently, this method has been applied to
produce other new results and simpler proofs of known results in
discrete geometry, see for
instance~\cite{KMS:spctdggkppt,Solymosi-Tao}.

The Guth-Katz polynomial partitioning method gives a nonlinear
decomposition of the Euclidean space, which plays a role analogous to
cuttings or trapezoidal decompositions in the more classical
Clarkson-Shor type divide-and-conquer arguments for such problems, see
for instance \cite{CEGSW}.

It can be summarized in the result below.  For a polynomial $g\in
\R[x_{1},\dots, x_{d}] $, we denote by $V(g)$ its zero zet in
$\C^{d}$ and, for a finite set $\cQ$, we denote by $\card(\cQ)$ its
cardinality.

\begin{theorem}[Guth and Katz \cite{GK:eddpp}]
\label{thm:Guth-Katz}
Let $d\ge 1$ and $\mathcal{P} \subset \R^d$  be a finite subset. Given $\ell\ge 1$,
there is a nonzero polynomial $g \in \R[x_{1},\dots, x_{d}] $ of
degree bounded by $\ell$ such that, for each connected component $C$
of $\R^d \setminus V(g)$,
\begin{equation*}
  \card (\mathcal{P} \cap C)  = O_{d}\Big(
  \frac{\card(\mathcal{P})}{\ell^{d}}\Big),  
\end{equation*}
where the implicit constant in the $O$-notation depends only on $d$. 
\end{theorem}

When applying this result in a concrete situation, one needs to
couple it with a suitable bound for the number of connected components
of the semi-algebraic set $\R^d \setminus V(g)$. This is provided by
the classical works of Ole{\u\i}nik, Petrovski{\u\i}, Milnor and Thom
on the Betti numbers of semi-algebraic varieties \cite{OP,Milnor2,T},
which allow to treat the points in $\cP$ outside the hypersurface
$V(g)$.

However, it is possible that many, or even all, of the points in
$\mathcal{P}$ are contained in this hypersurface. The points in
$\mathcal{P} \cap V(g)$ are not partitioned, and a separate argument
is needed for handling them.  The natural approach would be to apply a
polynomial partitioning theorem on $V(g)$ together with a suitable
bound for the number of connected components of the resulting
partition. After this step, it 
is 
also possible that many of the points
in $\mathcal{P} \cap V(g)$ are contained in the partitioning variety of
codimension 2. Then one would like to apply a partitioning theorem on this
variety, and so on.

To make this strategy work efficiently, one needs a polynomial
partitioning theorem on varieties.  For hypersurfaces, such a result
has been achieved independently by Zahl \cite{Zahl:ibnpsitd} and by
Kaplan, Matou{\v{s}}ek, Sharir and Safernov\'a \cite{KMSS:udtd}, and
applied to  incidence problems in $\R^{3}$. Extending it to
varieties of arbitrary codimension has been identified as a major
obstacle to apply the polynomial partitioning method to incidence
problems in dimension
$d\ge 4$, see for instance the discussion in \cite[\S 3]{KMSS:udtd}.
Our main objective in this paper is to present such a
result for irreducible varieties of codimension two. 

Given an irreducible algebraic variety $X\subset \C^{d}$ we denote by
$\dim(X)$ and $\deg(X)$ its dimension and degree, respectively.  We
also denote by $\delta(X)$ the minimal integer $\delta\ge1$ such that
$X$ is an irreducible component of the zero set of a family of
polynomials of degree bounded by $\delta$. These invariants are
related by the inequalities~(Lemma \ref{lemm:2})
\begin{displaymath}
  \delta(X)\le \deg(X) \le \delta(X)^{d-\dim(X)}.
\end{displaymath}

The following is a simplified version of our polynomial
partitioning theorem (Theorem \ref{thm:polynomial-partitioning}). 

\begin{theorem}
\label{thm:3} 
Let $d\ge 1$ and  $X \subset \C^d$ an irreducible variety of
codimension at most two. Let $\mathcal{P} \subset \R^{d}\cap X$ be a
finite subset and $\ell \geq 6\, 
d \,\delta(X)$.  Then there is a polynomial $g \in \R[x_{1},\dots,
x_{d}]$ of degree bounded by $\ell$ with $\dim(X\cap V(g)) =\dim(X)-1$
such that, for each connected component $C$ of $\R^{d}\setminus V(g)$,
\begin{equation*}
  \card( \mathcal{P} \cap C)  = O_{d}\Big(
  \frac{\card(\mathcal{P})}{\deg(X) \ell^{\dim(X)}}\Big).  
\end{equation*}
\end{theorem}

When $X=\C^{d}$, the invariant $\delta(X)$ is equal to 1 whereas,  when
$X$ is a hypersurface, it coincides with $\deg(X)$. Hence,
Theorem \ref{thm:3} reduces in these cases to Theorem
\ref{thm:Guth-Katz} and to the polynomial partitioning theorems in
\cite{Zahl:ibnpsitd, KMSS:udtd}, respectively.

As for the Guth-Katz theorem, the proof of this result is based on the
ham sandwich theorem obtained by Stone and Tukey from the Borsuk-Ulam
theorem. The new key ingredient is the systematic use of the upper and
lower bounds for Hilbert functions due to Chardin
\cite{Chardin:mfhcia} and Chardin and Philippon \cite{CP:ri}.

\begin{remark}
  \label{rem:2}
  The polynomial partition method also applies to problems in
  computational geometry, in particular to range searching with
  semi-algebraic sets.  Concurrently with this paper, Matou{\v{s}}ek
  and Pat{\'a}kov{\'a} have also obtained a polynomial partitioning theorem
on varieties \cite[Theorem 1.1]{MS:mppsrs}, focused on obtaining
   efficient range searching algorithms. 

   For irreducible varieties of codimension two, the
   Matou{\v{s}}ek-Pat{\'a}kov{\'a} partitioning theorem is
   quantitatively
   weaker than ours.  On the other hand, this result
   holds in a  more  general setting, since it can be applied to
   non-necessarily irreducible varieties of arbitrary dimension. This greater
   generality is important for their application to range searching.
\end{remark}

As a test case for Theorem \ref{thm:3}, we
consider the problem of bounding the number of point-hypersurface
incidences. Given a set  $\mathcal{P}$ of points of $\R^d$ and a set
$\cV$ of subvarieties of $\R^d$ or of $\C^{d}$, we denote by $I(\cP,\cV)$ their
number of incidences, that is, the number of pairs $(p,V)\in \cP\times\cV$ with  $p\in V$.

The following fundamental result was proved by Szemer\'edi and Trotter
in 1983, in response to a problem of Erd{\H o}s.

\begin{theorem}[Szemer\'edi and Trotter \cite{Szemeredi-Trotter}]
\label{thm:szemeredi-trotter}
Let $\cP$ be a set of $m$ points of $\R^{2}$ and $\cL$ a set of $n$
lines in $\R^2$. Then
\[
I(\cP,\cL)= O(m^{\frac23}n^{\frac23} + m + n).
\]
\end{theorem}

This theorem has led to an extensive study of
incidences of points and curves in the plane, and of points and
varieties in higher dimensions.  In particular, it was extended by
Pach and Sharir to incidences between points in the plane and curves
having a bounded degree of freedom \cite{Pach-Sharir}. Later on, Zahl
obtained an analogous result for the incidences between points in
$\R^{3}$ and algebraic surfaces having a bounded degree of freedom
\cite{Zahl:ibnpsitd}. A similar result was independently obtained by
Kaplan, Matou{\v{s}}ek, Sharir and Safernov\'a for the incidences
between points in $\R^{3}$ and unit spheres \cite{KMSS:udtd}.

We present the following bound for the number of incidences between points in
$\R^{4}$ and threefolds.

\begin{theorem}
\label{thm:main}
Given  $k,c \geq 1$,  let  $\mathcal{P}$ be a finite set of points of $\R^{4}$ and
$\cH$  a finite set of hypersurfaces of $\C^4$ satisfying the following conditions:
\begin{enumerate}
\item \label{item:4} the degrees of the hypersurfaces in $\cH$ are bounded
  by $c$; 
\item \label{item:5} the intersection of any 
four distinct
  hypersurfaces in $\cH$ is finite;
\item \label{item:6} for any subset of $k$ distinct points in $\cP$, the
  number of hypersurfaces in $\cH$ containing them is bounded by $c$.
\end{enumerate}
Set  $m=\card(\mathcal{P})$ and $n=\card(\cH)$. Then
\[
I(\mathcal{P},\cH) = O_
{k,c}(m^{1-\frac{k-1}{4k-1}} n^{1 - \frac{3}{4 k-1}} +m +n).
\]
\end{theorem}

This result is an application of Theorem \ref{thm:3} together with the
refined bounds for the number of connected components of a
semi-algebraic set due to Barone and Basu
\cite{Barone-Basu11a,Barone-Basu2013}. Our whole approach is strongly
inspired by the treatment of the unit distance problem in three
dimensions in \cite{Zahl:ibnpsitd,KMSS:udtd}.

Theorem \ref{thm:main} is a particular case of a conjectural bound for
the number of point-hypersurface incidences in $\R^{d}$ (Conjecture
\ref{conj:main-conjecture}).  Related with this, we propose two
further conjectures: a generalization of our polynomial partitioning
theorem to varieties of arbitrary codimension (Conjecture
\ref{conj:4}) and a bound for the number of connected components of a
semi-algebraic set depending on the degree of that variety, instead of
the B\'ezout number of a set of defining equations
(Conjecture~\ref{conj:3}). If one can show that these two conjectures
are true, it would be an important step in proving Conjecture
\ref{conj:main-conjecture} {via} the polynomial partitioning
method.

\begin{remark}
  \label{rem:3}
  The results of this paper were announced in the talk
  \cite{Sombra:bhfpi} at the IPAM workshop ``Tools from algebraic
  geometry''. Shortly afterwards, a proof by Fox, Pach, Suk, Sheffer
  and Zahl of a weaker version of Conjecture
  \ref{conj:main-conjecture} with an extra factor $m^{\eps}$ was
  announced in Sheffer's blog \cite{Sheffer:idds} and eventually
  appeared in \cite{Sheffer-et-al}.
\end{remark}

\medskip \noindent {\bf Acknowledgments.} We thank 
Zuzana Safernov\'a/Pat{\'a}kov{\'a},
Micha Sharir, Noam Solomon and Joshua Zahl for useful discussions and
pointers to the literature.
We also thank the anonymous referees for their remarks and
corrections, which have significantly improved this paper.

Part of this work was done while the authors met at the Institute for
Pure and Applied Mathematics (IPAM) during the Spring 2014
research program ``Algebraic Techniques for Combinatorial and
  Computational Geometry''.

\section{Preliminaries on Hilbert functions and semi-algebraic
  geometry} \label{sec:prel-hilb-funct}

Throughout this paper, we denote by $\N$ the set of nonnegative
integers.  Bold letters denote finite sets or sequences of objects,
where the type and number should be clear from the context: for
instance, $\bfx$ might denote the group of variables
$\{x_1,\dots,x_d\}$ so that $\R[\bfx]$ denotes the polynomial ring
$\R[x_1,\dots,x_d]$.

Given functions $f,g\colon \N\to \N$, the Landau symbol $f=O(g)$ means
that there exists $c \ge 0$ such that $f(l)\le c\, g(l)$ for all $l\in
\N$.  If we want to emphasize the dependence of the constant $c$ on
parameters, say $d$   and $k$, we will write $f=O_{d,k}(g)$.

\subsection{Hilbert functions and degree of definition of
  varieties} \label{sec:hilb-funct-degr}

Let $\PP^{d}(\C)$ denote the $d$-dimensional projective space over the
complex numbers. For an equidimensional variety $X \subset
\PP^{d}(\C)$, we denote by $ \dim(X)$ and $\deg(X)$ its dimension and
degree, respectively.  Recall that the degree of $X$ is classically
defined as the number of points in the intersection of $X$ with a
generic linear subspace $H$ of dimension~$d-\dim(X)$.

When $X$ is a hypersurface, this variety is defined by a single squarefree
homogeneous polynomial $g\in \C[z_{0}, \dots, z_{d}]$, unique up a
scalar factor, and we have $\deg(X)= \deg(g)$. In the other extreme,
when $\dim(X)=0$, we have $\deg(X)=\# X$.

A basic property of the notion of degree of varieties is its behavior
with respect to intersections. In particular, it verifies the
following version of \emph{B\'ezout's inequality}~\cite[Example
8.4.6]{Fulton:IT}: let $X_{i}\subset \PP^{d}(\C)$, $i=1,\dots t$, be
equidimensional varieties and $Z_{j}$, $j=0, \dots, l$, the irreducible
components of the intersection $\bigcap_{j=1}^{t}X_{j}$. Then
\begin{equation}
  \label{eq:2}
  \sum_{j=0}^{l}\deg(Z_{j})\le \prod_{i=1}^{t}\deg(X_{i}).
\end{equation}
In particular, if $g_{1},\dots, g_{d}\in \C[z_{0}, \dots, z_{d}]$ is a
family of homogeneous polynomials whose zero set in $\PP^{d}(\C)$ is
finite, then the cardinality of this zero set is bounded
by~$\prod_{i=1}^{d}\deg(g_{i})$.

\begin{definition}
  \label{def:2}
  Let $X\subset \PP^{d}(\C)$ be an irreducible variety and $\delta\ge
  1$. We say that $X$ is \emph{partially defined at
    degree} $\delta$ if there are homogeneous polynomials
  $g_{1},\dots, g_{t}\in \C [z_{0},\dots, z_{d}]$ of degree bounded by
  $\delta$ such that $X$ is an irreducible component of the zero set
  in $\PP^{d}(\C)$ of these polynomials.  Equivalently, there is an
  open subset $U\subset \PP^{d}(\C)$ such that $X\cap U\ne \emptyset $
  and  the zero set in $U$ of $g_{1},\dots, g_{t}$ agrees
  with $X\cap U$.

We denote by $\delta(X)$ the
  \emph{degree of partial definition} of $X$, defined as
  the minimal integer $\delta\ge 1$ such that $X$ is
  {partially defined at degree}~$\delta$.
\end{definition}

The degree of a variety and its degree of partial
definition are related by the following inequalities. 

\begin{lemma}
  \label{lemm:2}
  Let $X\subset \PP^{d}(\C)$ be an irreducible variety . Then
\begin{equation*}
  \delta(X)\le \deg(X) \le \delta(X)^{d-\dim(X)}.
\end{equation*}
\end{lemma}

\begin{proof}
  We first prove the left inequality. Set $e=\dim(X)$ and identify the
  projective space $\PP^{e+1}(\C)$ with the linear subspace of
  $\PP^{d}(\C)$ defined by the equations $z_{e+2}=\dots=z_{d}=0$.
Let 
$L\subset \PP^{d}(\C)$ be a generic
linear subspace of dimension $d-e-2$.
By making a linear change in coordinates which keeps $z_{e+2},\ldots,z_d$ unchanged,
and changes only the coordinates $z_{0},\cdots,z_{e+1}$, we can assume that $L$ is defined by the equations $z_0 = \cdots = z_{e+1}=0$. 
Now consider the projection
\begin{displaymath}
  \pi_{L}\colon \PP^{d}(\C)\setminus L \longrightarrow 
  \PP^{e+1}(\C)
\end{displaymath}
defined, for a point $p\in \PP^{d}(\C)\setminus L$, by setting
$\pi_{L}(p)$ as the unique
point in the intersection of $\PP^{e+1}(\C)$ with the linear subspace
generated by $L$ and $p$.
In other words $\pi_L((z_0:\cdots:z_d)) = (z_0:\cdots:z_{e+1}:0:\cdots:0)$ for 
$(z_0:\cdots:z_d) \not\in L$.
  
   Then $\overline{\pi_L(X)}$, the closure of
the image of $X$ under this map, is a hypersurface of $ \PP^{e+1}(\C)$
of the same degree as $X$. This hypersurface is defined by a
homogeneous polynomial $f_L\in \C[z_{0},\dots,z_{e+1}]$ with
$$
\deg(f_{L}) = \deg(\overline{\pi_L(X)}) = \deg(X). 
$$
Then the 
polynomial $f_L$ considered as an element of the ring $\C[z_0,\ldots,z_d]$
is
a homogeneous polynomial of degree $\deg(X)$ defining a hypersurface
of $\PP^{d}(\C)$ which contains $X$.  By choosing sufficiently many
linear subspaces $L$ as above, one can construct a family of
homogeneous polynomials of degree $\deg(X)$ defining the variety
$X$. Hence $\delta(X) \le \deg(X)$, as stated.

For the right inequality, let $g_{1},\dots, g_{t}\in \C [z_{0},\dots,
z_{d}]$ be a family of homogeneous polynomials of degree $\le
\delta(X)$ having $X$ as an irreducible component of its zero set. By
taking generic linear combinations, we can suppose that
$t=d-e$. B\'ezout's inequality \eqref{eq:2} then implies that $\deg(X)
\leq \delta(X)^{d - e}$, proving the inequality. 
\end{proof}

\begin{remark}
\label{rem:deltavsdeg}
The degree of partial definition of a variety can by
much smaller than its degree. An example is provided by the
Grassmannian $\Gr(1,n)$, the space parametrizing lines in $\PP^n(\C)$,
included in the projective space $\PP\big(\bigwedge^{2}\C^{n+1}\big)$ via the
Pl\"{u}cker embedding.

The degree of this Grassmannian is
\[
\deg(\Gr(1,n)) =  \frac{1}{n-1}\binom{2n -2}{n},
\]
which clearly grows with $n$. On the other hand, this variety is cut
out by certain quadratic equations, called the Pl\"{u}cker
relations. Hence,
\[
\delta(\Gr(1,n)) =  2.
\]
\end{remark}

For irreducible varieties of codimension $2$, we have the
following sharpening of the second inequality in Lemma \ref{lemm:2}. 

\begin{lemma}
\label{lem:codim2}
 Let $X\subset \PP^{d}(\C)$ be an irreducible variety of
 codimension 2. Let $\delta_{1}\ge 1$ be the minimal degree of a
  hypersurface of $\PP^{d}(\C)$ containing $X$.
  Then
  \[
  \deg(X) \leq \delta_1 \delta(X).
  \]
\end{lemma}

 \begin{proof}
   Let $f$ be a homogeneous polynomial of degree $\delta_1\geq 1$
   vanishing on $X$. By the minimality assumption, this polynomial
   must be irreducible. Since $X$ is of codimension~2, there are two
   homogeneous poynomials $g_1,g_2$ of degree $\le \delta(X)$, having
   $X$ as an irreducible component of its zero set. By clearing common
   factors, we can also assume that one of these polynomials are
   coprime. Hence, at least 
   one of these
   these polynomials (say~$g_1$) is not
   divisible by $f$. Thus $X$ is also an irreducible component of the
   zero set of $f$ and $g_{1}$.  The lemma then follows from
   B\'ezout's inequality~\eqref{eq:2}.
 \end{proof}

Given a homogeneous ideal $I\subset
\C[z_{0},\dots, z_{d}]$, the quotient $ \C[z_{0},\dots, z_{d}]/I$ is a
graded $\C$-algebra. The \emph{Hilbert function} of $I$ is the function
$\Hf_{I}\colon \N\to \N$ given, for $\ell \in\N$, by the dimension of the
$\ell $-th graded piece of this quotient, that is
\begin{displaymath}
  \Hf_{I}(\ell )=\dim_{\C }\big(\C [z_{0},\dots, z_{d}]/I \big)_{\ell }. 
\end{displaymath}
By Hilbert's theorem, there is a polynomial $\Hp_{I}\in \Q[t]$ and an
integer $\ell _{0}\in \N$ with
\begin{displaymath}
  \Hf_{I}(\ell )= \Hp_{I}(\ell ) \quad \text{ for } \ell \ge \ell _{0}.
\end{displaymath}

For an  equidimensional  variety $X\subset \PP^{d}(\C)$,
we denote by $I(X)\subset
\C[z_{0},\dots, z_{d}]$ its defining ideal. Then $\Hp_{I(X)}$ is a
polynomial of degree $\dim(X)$ and leading coefficient equal to
the quotient ${\deg(X)}/{\dim(X)!}$.

In Theorem \ref{thm:1} below, we collect the upper and lower bounds
for Hilbert functions that we will use later on. Because of our
applications, we restrict to ideals coming from irreducible projective
varieties, although these bounds are valid in greater generality.
Recall that binomial coefficients are defined, for $i,n\in \Z$, by
\begin{displaymath}
  {n\choose i}=
  \begin{cases}
\displaystyle    \frac{n!}{i!(n-i)!} & \text{ if } 0\le i\le n, \\
0 & \text{ otherwise}. 
  \end{cases}
\label{sec:hilbert-functionsc}
\end{displaymath}

\begin{theorem}
  \label{thm:1}
Let $X\subset \PP^{d}(\C)$ be an irreducible variety  of dimension $e\ge 0$. 
\begin{enumerate}
\item \label{item:1} For $\ell \ge 0$, 
\begin{displaymath}
  \Hf_{I(X)}(\ell )\le \deg(X) {\ell +e\choose e}.
\end{displaymath}
\item \label{item:3} For $\ell \ge (d-e)(\delta(X)-1)+1$,
\begin{displaymath}
  \Hf_{I(X)}(\ell )\ge \deg(X) {\ell - (d-e)(\delta(X)-1) +e\choose e}.
\end{displaymath}
\end{enumerate}
\end{theorem}

\begin{proof}
  The upper bound in \eqref{item:1} is \cite[Th\'eor\`eme on page
  306]{Chardin:mfhcia} applied to the base field $\C$ and the ideal
  $I(X)$.  Similar bounds can also be derived from  \cite{Nesterenko:ecfpd} or
  \cite[Proposition 2.11]{Sombra:bhfpi}.

  The upper bound in \eqref{item:3} is a particular case of
  \cite[Corollaire 3]{CP:ri}. Indeed, let $g_{1},\dots, g_{t}\in \C
  [z_{0},\dots, z_{d}]$ be a family of homogeneous polynomials of
  degree $\le \delta(X)$ having $X$ as an irreducible component of its
  zero set. By taking generic linear combinations, we can suppose
  without loss of generality that $t=d-e$. Consider the ideals
  $I=(g_{1},\dots, g_{d-e})$ and $J=I(X)$. Following the notation in
  page 476 of loc. cit., the ideal $I^{\langle d-e\rangle}$ is defined
  as the intersection of the isolated primary ideals of $I$ of
  codimension $d-e$. Hence $J\subset I^{\langle d-e\rangle}$.
   We can
  then apply \cite[Corollaire 3]{CP:ri} to these ideals. In the
  notation of this result, $m=d-e$, $d_{i}=\deg(g_{i})$ for
  $i=1,\dots, d-e$, and $r=d-e$. This result then implies that, for
  $\ell\ge \sum_{i=1}^{d-e}\deg(g_{i}) - (d-e)$,
\begin{displaymath}
  \Hf_{I(X)}(\ell )\ge \deg(X) {\ell +d- \sum_{i=1}^{d-e}\deg(g_{i})\choose e},
\end{displaymath}
which gives the lower bound in \eqref{item:3}.
\end{proof}

The following result is a consequence of Theorem
\ref{thm:1}\eqref{item:1}, and appears as a particular case of
\cite[Corollaire 3]{Chardin:mfhcia}. We include its proof, for the
convenience of the interested reader.

\begin{proposition}
  \label{prop:1}
  Let $X\subset \PP^{d}(\C)$ be an irreducible variety of
  codimension 2. Then there are coprime polynomials $f_{1},f_{2}\in
  I(X) $ such that
\begin{displaymath}
  \deg(f_{1})\deg(f_{2})\le d(d-1)\deg(X). 
\end{displaymath}
\end{proposition}

\begin{proof}
  Set $D=\deg(X)$.  By Theorem \ref{thm:1}\eqref{item:1}, for $\ell\ge
  0$,
\begin{equation}
  \label{eq:6}
\Hf_{I(X)}(\ell)  \le D{\ell+d-2\choose d-2}.  
\end{equation}
We have $ \Hf_{\{0\}}(\ell)= \dim_{\C}\C[\bfz]_{\ell} ={\ell+d\choose
  d} $. This implies that, for $\ell_{1}=\lfloor
(d(d-1)D)^{1/2}\rfloor$,
\begin{displaymath}
  \Hf_{I(X)}(\ell_{1}) <\Hf_{\{0\}}(\ell_{1}).
\end{displaymath}
Hence, there is a homogeneous polynomial $f_{1}\in I(X)\setminus
\{0\}$ with $\deg(f_{1}) \le \ell_{1}$. We take $f_{1}$ of minimal
degree. Since the variety $X$ is irreducible, this polynomial has to
be irreducible too.

By the exact sequence
\begin{displaymath}
  0\longrightarrow \C[\bfz]\stackrel{\times f_{1}}{\longrightarrow} \C[\bfz] \longrightarrow
  \C[\bfz]/(f_{1}) \longrightarrow 0,
\end{displaymath}
 the Hilbert function of the principal ideal $(f_{1})$ is given by 
\begin{multline} \label{eq:8} \Hf_{(f_{1})}(\ell)= \dim_{\C}\C[\bfz]_{\ell} -
\dim_{\C}\C[\bfz]_{\ell-\ell_{1}}\\= 
{\ell+d\choose d} -
  {\ell-\ell_{1}+d\choose d}  =
  \sum_{j=0}^{\ell_{1}-1}{\ell-j+d-1\choose d-1}.
\end{multline}
Using this, one can verify  that, for $\ell_{2}=\max \big\{ \lfloor
(d(d-1)D)^{1/2}\rfloor, \lfloor d(d-1) D/\ell_{1}\rfloor\big\}$, 
\begin{displaymath}
  \Hf_{I(X)}(\ell_{2}) <\Hf_{(f_{1})}(\ell_{2}).
\end{displaymath}
Hence, there is a homogeneous polynomial $f_{2}\in I(X)\setminus
(f_{1})$ with $\deg(f_{2}) \le \ell_{2}$. Hence,  the polynomials  $f_{1},f_{2}$  are
coprime and satisfy
\begin{displaymath}
  \deg(f_{1})\deg(f_{2})\le \ell_{1}\ell_{2}\le d(d-1)D,
\end{displaymath}
as stated. 
\end{proof}

The next result gives a lower bound for the Hilbert function of the
 ideal of a variety $X$ of codimension two. For $\ell\ge0$, it exhibits three different behaviors, depending on the
codimension of the zero set of the graded part $I(X)_{\ell}$.

\begin{proposition}\label{prop:2}
  There is a constant $c=c(d) >0$ with the following property.  Let
  $X\subset \PP^{d}(\C)$ be an irreducible variety of codimension
  2. Let $\delta_{1}\ge 1$ be the minimal degree of a hypersurface of
  $\PP^{d}(\C)$ containing $X$ and set $\delta_{2}=\delta(X)$. Then
\begin{displaymath}
  \Hf_{I(X)}(\ell )\ge 
  \begin{cases}
   c\,  (\ell+1) ^{d}+1 & \text{ if } 1\le \ell \le \delta_{1}-1, \\
   c\,  \delta_{1}(\ell+1) ^{d-1}+1 & \text{ if } \delta_{1}\le \ell \le \delta_{2}-1, \\
   c\,  \delta_{1}\delta_{2}(\ell+1) ^{d-2}+1 & \text{ if } \delta_{2}\le \ell.
  \end{cases}
\end{displaymath}
\end{proposition}

\begin{proof}
  We have $\delta_{1}=\min\{\ell \ge 0 \mid I(X)_{\ell}\ne \{0\}\} $.
  Hence, for $ 1\le \ell \le \delta_{1}-1$,
\begin{equation*}
\Hf_{I(X)}(\ell)= \Hf_{\{0\}}(\ell)= \dim_{\C}\C[\bfz]_{\ell}=
{\ell+d\choose d}. 
\end{equation*}
Thus
\begin{equation} \label{eq:19}
\Hf_{I(X)}(\ell) \ge    c_{1} (\ell+1) ^{d}+1
\end{equation}
for a suitable constant $c_{1}>0$ depending only on $d$, giving the
first lower bound.

Let  $f_{1}$
be a nonzero polynomial in $I(X)$ of degree $\delta_{1}$. 
By the minimality property of $\delta_1$, this polynomial  must be
irreducible. 
We have 
 $\delta_{2}=\min\{\ell \ge \delta_{1} \mid I(X)_{\ell}\ne (f_{1})_{\ell}\} $.
Hence, for  $ \delta_1 \leq  \ell \leq \delta_2 -1$,
\[
\Hf_{I(X)}(\ell)=   \Hf_{(f_{1})}(\ell)=  \sum_{j=0}^{\delta_{1}-1}{\ell-j+d-1\choose d-1},
\]
where the second equality comes from \eqref{eq:8}.
It follows that, for $\delta_{1}\le \ell \le
\delta_{2}-1$, 
\begin{equation}\label{eq:22}
  \Hf_{I(X)}(\ell)\ge 
 \frac{\delta_{1}}{2(d-1)!}\Big(\ell- \frac{\delta_{1}}{2}\Big)^{d-1}+1 
\ge c_{2} \delta_{1}(\ell+1) ^{d-1}+1 
\end{equation}
for another constant $c_{2}=c_{2}(d)>0$. 

Finally, we consider the case when $\ell\ge \delta_{2}$. When $\ell\le
2   (\delta_{2}-1)$, we deduce from~\eqref{eq:22} that
\begin{equation}
\label{eq:23}
\Hf_{I(X)}(\ell)\ge  \Hf_{I(X)}(\delta_{2}-1) 
\ge c_{2} \delta_{1}\delta_{2}^{d-1}+1 \ge 
   c_{3} \delta_{1}\delta_{2}(\ell+1) ^{d-2}+1 .
\end{equation}
 Proposition \ref{prop:1} implies that
$ d(d-1)\deg(X)\ge \delta_{1}\delta_{2}$. Hence, for $\ell\ge
2(  \delta_{2}-1) +1$,  Theorem
\ref{thm:1}\eqref{item:3} implies that
\begin{equation}
  \label{eq:5}
  \Hf_{I(X)}(\ell)\ge  \deg(X) {\ell - 2(\delta_{2}-1)+d-2\choose d-2} \ge
  c_{4}\delta_{1}\delta_{2}(\ell+1)^{d-2}+1.
\end{equation}
The result follows from \eqref{eq:19},
\eqref{eq:22}, \eqref{eq:23} and \eqref{eq:5} by taking
$c=\min_{i}c_{i}$.
\end{proof}

\subsection{Connected components of semi-algebraic sets}

As explained in the introduction, the polynomial partitioning method
has to be coupled with bounds for the number of connected components
of semi-algebraic sets. When partitioning the Euclidean space $\R^{d}$, the
appropriate bound follows from the
Ole{\u\i}nik-Petrovski{\u\i}-Milnor-Thom's bounds for the Betti
numbers of a semi-algebraic set \cite{OP,Milnor2,T}: with
notation as in Theorem \ref{thm:Guth-Katz}, the number of connected
components of $\R^d \setminus V(g)$ is bounded by $ \ell(2\ell
-1)^{d-1} = O(\ell^{d}).$

In our situation, we will need the Barone-Basu bound for the number of
connected components, with a refined dependence on the degrees of the
polynomials \cite{Barone-Basu11a, Barone-Basu2013}.  We recall a
simplified version of this result in Theorem \ref{thm:2} below.

Given $f_{1},\dots, f_{e}\in \R[x_{1},\dots, x_{d}]$, we denote by
$V(f_{1},\dots, f_{e})$ its zero set in $\C^d$. For a variety
$X\subset \C^{d}$, we denote by $X(\R)=X\cap \R^{d}$ its set of real
points.  For a semi-algebraic subset $S \subset \R^d$, we denote by
$\Cc(S)$ the set of connected components of $S$. The $0$-th Betti
number $\betti(S)$ coincides with the cardinality of the set $\cc(S)$.

\begin{theorem}\label{thm:2}
  There is a constant $c=c(d)$ with the following property. Let  $f_{1},\dots, f_{e}, g\in \R[x_{1},\dots,
  x_{d}]$ with
  $\deg(f_1) \leq \cdots \leq \deg(f_e)\le \deg(g)$ such that
  $\dim(V(f_{1},\dots, f_{i}))= d-i$ for $i=1,\dots, e$. Then both
  \begin{displaymath}
   \betti(V(f_{1},\dots, f_{e})(\R)\setminus V(g))  \quad \text{ and }
   \quad 
   \betti(V(f_{1},\dots, f_{e},g)(\R)) 
  \end{displaymath}
  are bounded by $ c \deg(f_{1})\dots \deg(f_{e}) \deg(g)^{d-e}.$
\end{theorem}

\begin{proof}
  The semi-algebraic set $V(f_{1},\dots, f_{e})(\R)\setminus V(g) $ is the
  union of the realization of the sign conditions $\pm1$ of $g$ on
  $V(f_{1},\dots, f_{e})(\R)$. Similarly, $V(f_{1},\dots,
  f_{e},g)(\R)$ is the realization of the sign condition $0$ of $g$ on
  the same real algebraic variety. 

  The result  follows from \cite[Theorem 4]{Barone-Basu2013}
  and
  the fact that $\dim(V(f_{1},\dots, f_{i}))$ bounds from above the
  dimension of the semi-algebraic set $V(f_1,\ldots,f_i)(\R)$, see
  Remark 1.10 in 
  loc. cit.
    \end{proof}

We will also need the technical result below.  Given $p\in \R^{d}$
and $r > 0$, we denote by $\ball(p,r)$ the open ball in $\R^d$ with
center $p$ and radius $r$.  Given a variety $W\subset \C^{d}$ and a hypersurface
$H\subset \C^{d}$, we denote by $B(W,H)$ the subset of $W(\R)$
 of points $p \in W(\R)$ having an open neighborhood, in
the Euclidean topology of $W(\R)$,
contained in $H$. We also set $G(W,H)= W(\R) \setminus B(W,H)$.

\begin{proposition}
\label{prop:A}
Let $W \subset \C^4$ be a variety and $H, K \subset \C^{4}$ two
hypersurfaces. Let $b\in \R[x_{1},x_{2},x_{3}, x_{4}]$ be a polynomial
defining $H$.  Then there exists $\varepsilon_{0}>0$ such that, for
all $0<\eps\le \eps_{0}$ and any variety $\wt{W} \subset \C^{4}$
containing $W$, the number of connected components $C$ of $\R^4
\setminus K$ such that $C \cap G(W,H) \cap H \neq \emptyset$ is
bounded by
\[
\betti((\wt{W}(\R) \cap V(b^2 - \eps)) \setminus K).
\]
\end{proposition}

\begin{proof}
  Consider the set of connected components
\begin{displaymath}
\mathcal{C} = \{ C \in \Cc(\R^4 \setminus K) \mid C \cap
  G(W,H) \cap H \neq \emptyset \}.   
\end{displaymath}
For each $C \in \mathcal{C}$ choose a point $p_C \in C \cap G(W,H)
\cap H$. Since $\mathcal{C}$ is a finite set, the set of points
$\{p_C\}_{C \in \mathcal{C}}$ is finite.
For each $C \in \mathcal{C}$ and $r>0 $, consider also the semi-algebraic
set given by
\begin{displaymath}
 U_r(p_C)= \ball(p_{C},r) \cap (W(\R) \setminus K). 
\end{displaymath}
By the definition of $G(W,H)$, the set $U_r(p_c)$ is not contained in
$H$.  Semi-algebraic sets are locally contractible because of their
local conical structure, see for instance \cite[Theorem
5.48]{BPRbook2}.  Hence, there exists $r_C > 0$ such that, for all $ 0
< r \leq r_C$, the  set $U_r(p_C)$ is contractible and,
in particular, connected. Set $r_0 = \min_{C} r_C$.

Choose also $q_C \in U_{r_0}(p_C) \setminus H$ and a semi-algebraic
path 
$\gamma_C:[0,1]\rightarrow U_{r_{0}}(p_C)$ 
with $\gamma(0)=p_C$ and
$\gamma(1) = q_C$. We have that $b^2(q_C) >0$ because $q_{C}\notin
H$. We set
$\eps_0 = \min_{C} b^2(q_C)$. 

By the intermediate value theorem, for all $C \in \mathcal{C}$ and $0
<\eps \leq \eps_0$, there exists $0<t_{C} \le 1$ such that $b^2(z_{C})
= \eps$ with $z_{C}= \gamma_C(t_C)$. By construction,
\begin{displaymath}
z_C \in ({W}(\R) \cap V(b^2 - \eps))\setminus K \subset (\wt{W}(\R) \cap V(b^2 -
\eps))\setminus K.    
\end{displaymath}
Moreover, $z_{C}\in C$ because this point is connected by a path to
$p_{C}$.  For $C,C' \in \mathcal{C}$ with $ C\neq C'$, the points
$z_C$ and $ z_{C'}$ belong to distinct connected components of
$\wt{W}(\R) \cap V(b^2 - \eps) \setminus K$.  Hence, the map $C\mapsto
z_{C}$ induces an injection between the set of connected components
$\cC$ and $\cc(\wt{W}(\R) \cap V(b^2 - \eps)) \setminus K$, which
proves the proposition.
\end{proof}

In connection with the application of the polynomial partitioning
theorem to incidence problems in higher dimensions, we propose the
following conjectural bound for the number of connected components of
a semi-algebraic set in terms of the degree of the variety instead of
the B\'ezout number of a set of defining equations.

\begin{conjecture} 
\label{conj:3}
Let $X \subset \C^d$ be an irreducible variety and  $g\in \R[x_{1},\dots, x_{d}]$  a polynomial of degree
$\ell\ge \delta(X)$. Then there exists a variety $Y\subset \C^{d}$
containing $X$ as an irreducible component such that
  \begin{displaymath}
    \betti(Y(\R)\setminus V(g)) \quad \text{ and } \quad \betti(Y(\R)\cap V(g))
  \end{displaymath}
are bounded by $O_{d}(\deg(X)  \ell^{\dim(X)})$.
\end{conjecture}

When $X$ is an irreducible variety of codimension 2, this statement
follows easily from Proposition \ref{prop:1} and Theorem
\ref{thm:2}. In this case, the variety $Y$ is given, in the notation
of Proposition \ref{prop:1}, by the zero set of $f_{1}$ and $f_{2}$.

\section{Partitioning finite sets on varieties} 
\label{sec:part-finite-sets}

Given a set of points $\cP\subset \R^{d}$ and a set of polynomials
$\cG\subset \R[x_{1},\dots, x_{d}]$, for each choice of signs
$\gamma\in \{\pm1\}^{\cG}$ we put
  \begin{equation}
 \label{eqn:sign-pattern}
\cP(\gamma) =\{p\in \cP \mid 
\gamma_{g}
g(p) >0 \text{ for
  all } g\in \cG\},
 \end{equation}
If the  set of polynomials $\cG$ is clear from the context, then we say
 that $\cP(\gamma)$ is \emph{realized} by 
 $\gamma$.

 Given $g\in \R[x_{1},\dots,x_{d}]\setminus \{0\}$, we denote by
 $\irr(g) \subset \R[x_{1},\dots, x_{d}]$ a complete and irredundant
 set of irreducible factors of $g$. These irreducible factors are
 unique up to scalars  in $\R^{\times}$. To fix their
 indeterminacy, we choose them to be monic with respect to some fixed
 monomial order on $\R[x_{1},\dots, x_{d}]$. With this convention, the
 set $\irr(g)$ is uniquely defined and
  \begin{displaymath}
    g=\lambda\prod_{q\in \irr(g)}q^{e_{q}}
  \end{displaymath}
with $\lambda\in \R^{\times}$ and $e_{q}\in\N$. 

For $\ell\ge 0$, we denote by $\R[x_{1},\dots, x_{d}] _{\leq \ell}$
the linear subspace of $\R[x_{1},\dots, x_{d}]$ of polynomials of
degree bounded by $\ell$.  Recall that, for a variety $X\subset
\C^{d}$, we denote by $X(\R)=X\cap \R^{d}$ its set of real points.

We state and prove our polynomial partitioning theorem in terms of
sign conditions. For convenience, we state it for varieties of
codimension \emph{at most} two, even though we  prove it only when
the codimension is two. The cases when the codimension is smaller are
simpler and can be proven as in \cite{GK:eddpp, Zahl:ibnpsitd,
  KMSS:udtd}.

\begin{theorem} 
\label{thm:polynomial-partitioning} 
Let $X\subset \C^d$ be an irreducible variety of codimension at most
two, $\mathcal{P} \subset X(\R)$ a finite subset and $\ell \geq 6 \,
d\, \delta(X)$.  Then there exists $g \in \R[x_{1},\dots, x_{d}]
_{\leq \ell}$ with $\dim(X\cap V(g)) =\dim(X)-1$ such that, for each $
\gamma \in \{\pm 1\}^{\irr(g)}$,
\begin{equation*}
  \card( \mathcal{P}(
  \gamma
  ))  = O_{d}\Big(
  \frac{\card(\mathcal{P})}{\deg(X) \ell^{\dim(X)}}\Big).  
\end{equation*}
\end{theorem}

\begin{remark} \label{rem:1} Let $S\subset \R^{d}$ be an arbitrary
  subset. For each connected component $C$ of $S\setminus V(g)$, the
  set $\cP\cap C$ is contained in a set of the form $\cP(\gamma)$ with
  $\gamma\in\{\pm1\}^{ \irr(g)}$. Hence, Theorem \ref{thm:3} in the introduction
  follows  from Theorem \ref{thm:polynomial-partitioning}
  above by choosing $S=\R^{d}$. 
\end{remark}

Given $\ell \ge 0$, we denote by $v_{\ell } $ the Veronese embedding
$\C^{d}\hookrightarrow \C^{{\ell +d\choose d}-1}$ given, for a point $\bfp=(p_{1},\dots, p_{d})\in \C^{d}$, by 
\begin{equation}
  \label{eq:1}
  v_{\ell }(\bfp)= (\bfp^{\bfa})_{\bfa}
\end{equation}
where $\bfa=(a_{1},\dots, a_{d}) \in \N^{d}$ runs over all nonzero
vectors of length $|\bfa|=\sum_{i}a_{i}$ bounded by $\ell$, and where
$\bfp^{\bfa}$ denotes the monomial $p_{1}^{a_{1}}\dots p_{d}^{a_{d}}$.
We also denote by $\iota$ the standard inclusion $\C^{d}\to \PP^{d}(\C)$
given by
\begin{equation*}
  \iota(\bfp)=(1:p_{1}:\dots:p_{d}). 
\end{equation*}

For a subset $E\subset \R^{d}$, we denote by $\aff(E)$ the smallest
affine subspace 
(or flat) 
of $\R^{d}$ containing $E$.  We also denote by
$I({\iota(E)}) \subset \C[z_{0},\dots, z_{d}]$ the homogeneous ideal
of polynomials vanishing identically on the subset
${\iota(E)}\subset\PP^{d}(\C)$.


\begin{lemma}
  \label{lemm:1}
With notation as above, let $E\subset \R^{d}$ be a subset and $\ell\ge 0$. Then
\begin{displaymath}
\dim_{\R}(\aff(v_{\ell }(E)))= \Hf_{I({\iota(E)})}  (\ell )-1.
\end{displaymath}
\end{lemma}

\begin{proof}
The ideal $I({\iota(E)})$ is
  generated over $\R[\bfz]$, because it is defined by the vanishing of
  a set of real points.  Setting $I=I({\iota(E)})\cap
  \R[\bfz]$,  we have 
\begin{equation}
  \label{eq:17}
\Hf_{I({\iota(E)})}  (\ell ) = \dim_{\C}\big( \C[\bfz]/I({\iota(E)})\big)_{\ell}
= \dim_{\R}( \R[\bfz]_{\ell })-\dim_{\R}(I_{\ell }),
\end{equation}
where $I_{\ell}$ denotes the $\ell$-th graded part of $I$. 

Consider the Euclidean space $\R^{\ell +d\choose d}$ with coordinates indexed by the vectors of
$\N^{d+1}$ of length equal to $\ell $, and  the pairing defined by
\begin{equation}
  \label{eq:15}
  \R[\bfz]_{\ell }\times \R^{\ell +d\choose d} \longrightarrow \R ,\quad
  \bigg(\sum_{|\bfb|=\ell }\alpha_{\bfb}\bfz^{\bfb},\bfw\bigg)\longmapsto  \sum_{\bfb}\alpha_{\bfb}\bfw_{\bfb}, 
\end{equation}
where $\bfb$ runs over all vectors of $\N^{d+1}$ of length $\ell $.

Consider the subset $\{1\}\times v_{\ell }(E) \subset \R^{\ell
  +d\choose d}$, with $1\in \R$ and $v_{\ell }(E) $ the image of $E$
under the Veronese embedding \eqref{eq:1}. 
The graded part $I_{\ell}$ coincides with the annihilator of this
subset with respect to the pairing \eqref{eq:15}. Since
$I_{\ell}$ is a linear subspace, it also coincides with the
annihilator of the linear span in $ \R^{\ell +d\choose d} $ of this
subset. Denote by $\lin(\{1\}\times v_{\ell }(E))$ this linear span,
which is a linear space  containing $ \{1\}\times \aff( v_{\ell
}(E))$ as an affine hyperplane. Hence
\begin{equation}\label{eq:18}
 \dim_{\R}( \R[\bfz]_{\ell })-\dim_{\R}(I_{\ell })
= \dim_{\R}( \lin(\{1\}\times
v_{\ell }(E))) 
 = 
\dim_{\R}(\aff(v_{\ell }(E)))+1.
\end{equation}
The result then follows from \eqref{eq:17} and \eqref{eq:18}.
\end{proof}

\begin{proof}[Proof of Theorem \ref{thm:polynomial-partitioning}]
We assume that $X$ is of codimension 2. 
  Let $\delta_{1}\ge 1$ be  the minimal degree of a hypersurface of
  $\PP^{d}(\C)$ containing $X$ and set also
  $\delta_{2}=\delta(X)$. Let $\eta\ge \delta_{2}$ be an integer to be fixed later on.

  Let $c=c(d)$ be the constant in Proposition \ref{prop:2} and set
  $c_{1}=\min\{c, 2^{-d}\}$. Put
\begin{equation*}
s_{0} =   \log(c_{1}\delta_{1}^{d}), \quad
s_{1} =  \log(c_{1}\delta_{1}\delta_{2}^{d-1}), \quad
t =  \lfloor \log(c_{1}\delta_{1}\delta_{2}\eta^{d-2})\rfloor
\end{equation*}
and 
  \begin{equation*}
    \ell_{i}=
    \begin{cases}
 \lfloor      (c_1^{-1}2^{i})^{\frac{1}{d}}  \rfloor & \text{ for } 
      0 \leq i < s_{0}, \\ 
 \lfloor      (c_1^{-1}\delta_{1}^{-1}2^{i})^{\frac{1}{d-1}}  \rfloor & \text{ for } 
      s_{0} \leq i < s_{1}, \\ 
 \lfloor      (c_1^{-1}\delta_{1}^{-1}\delta_{2}^{-1}2^{i})^{\frac{1}{d-2}}  \rfloor & \text{ for } 
      s_{1} \leq i \leq t, \\ 
    \end{cases}
  \end{equation*}
  where $\lfloor \cdot \rfloor$ denotes the floor function. We verify
  that the following conditions hold:
    \begin{equation}
      \label{eq:30}
      \begin{aligned}
& \text{ if }   0 \leq i < s_{0} \text{ then } 1\le \ell_{i}\le
  \delta_{1}-1,\\
& \text{ if }   s_{0} \leq i < s_{1} \text{ then } \delta_{1}\le \ell_{i}\le \delta_{2}-1,\\
& \text{ if }   s_{1} \leq i \le
  t \text{ then } \delta_{2}\le \ell_{i}\le \eta      .
      \end{aligned}
    \end{equation}

 Let $v_{\ell_{i}}$ be the Veronese map of degree $\ell_{i}$ as in
 \eqref{eq:1} and set $A_{i}\subset \R^{{\ell_{i}+e\choose e}} - 1$ for
 the affine hull of the image of $X(\R)$ under $v_{\ell_{i}}$. Let $I({\iota(X(\R))})$ be the ideal of polynomials vanishing on
  the image under $\iota$ of the set of real points of $X$.  
By Lemma \ref{lemm:1},
  \begin{equation}\label{eq:4}
    \dim_{\R}(A_{i}) = \Hf_{I({\iota(X(\R))})}(\ell_{i})-1.
  \end{equation}
Since
  ${\iota(X(\R))} \subset \iota(X)$, we have that $I({\iota(X(\R))})
  \supset I( \iota(X))$ and so
\begin{equation} \label{eq:9}
  \Hf_{I({\iota(X(\R)}))}(\ell_{i})\le   \Hf_{I(\iota({X}))}(\ell_{i}). 
\end{equation}

We consider first the case when \eqref{eq:9} is an equality for all
$i$.  Since the affine variety $X$ is irreducible and partially
defined at degree $\delta(X)$, the same holds for $\ov {\iota(X)}$,
the Zariski closure of $\iota(X)$ in projective space. It follows from
Proposition \ref{prop:2} and the conditions in \eqref{eq:30} that
\begin{equation}
  \label{eq:21}
\Hf_{I( {\iota(X(\R))})}(\ell_{i})\ge 2^{i}+1, \quad i=0,\dots, t.
\end{equation}

As in the Guth-Katz polynomial partitioning, we
will inductively subdivide the set of points $\cP$. We start with
$\cC_{0}=\{\cP\}$. Having constructed $\cC_{i}$ with at most $2^{i}$
sets, we apply the ham sandwich theorem to the image of these sets
under the map $v_{\ell_{i}}$. These images lie in $A_{i}$ and, by
\eqref{eq:4} and \eqref{eq:21}, this is an affine space of dimension
$\ge 2^{i}$. Hence, there is a nonzero linear form on $A_{i}$ that
bisects each of these images or, equivalently, there is a polynomial
$g_{i}\in \R[x_{1},\dots, x_{d}]_{\le \ell_{i}}$ bisecting each of the
sets in $\cC_{i}$.

  For each  $\cQ\in \cC_{i}$, we put $\cQ^{+}$ and $\cQ^{-}$ for
  the sets of points of $\cQ$ at which $g_{i}>0$ and $g_{i}<0$,
  respectively. We then put
  \begin{displaymath}
    \cC_{i+1}=\bigcup_{\cQ\in \cC_{i}} \{\cQ^{+}, \cQ^{-}\}.
  \end{displaymath}
Hence, each of the sets in $\cC_{t}$ has 
cardinality bounded by $2^{-t}{\card( \cP)}$.

Set $g=\prod_{i=0}^{t}g_{i}$. 
To bound the degree of $g$, we write  $  \deg(g) =S_{0}+S_{1}+S_{2}$ with 
\begin{equation*}
  S_{0}= \sum_{0\le i<s_{0}}\ell_{i}, \quad
  S_{1}= \sum_{s_{0}\le i<s_{1}}\ell_{i}, \quad
  S_{2}= \sum_{s_{1}\le i\le t}\ell_{i}.
\end{equation*}
We have that 
\begin{equation*}
  S_{0}\le \sum_{i=0}^{s_{0}-1} (c_{1}^{-1}2^{i})^{\frac{1}{d}}\le
  c_{1}^{-\frac{1}{d}} \frac{2^{\frac{s_{0}+1}{d}}-1}{2^{\frac{1}{d}}-1}
  \le \frac{2^{\frac{1}{d}}}{2^{\frac{1}{d}}-1} \delta_{1}.
\end{equation*}
Similarly, one can verify that
\begin{equation*}
S_{1}  \le \frac{2^{\frac{1}{d-1}}}{2^{\frac{1}{d-1}}-1} \delta_{2}
 \quad \text{ and }
\quad S_{2}\le  \frac{2^{\frac{1}{d-2}}}{2^{\frac{1}{d-2}}-1} \eta.
\end{equation*}
Using that $d\ge 3$ and $\delta_{1}\le \delta_{2}$, we deduce that $
\deg(g) \le 4\, d\, \delta_{2} +  2\, d\, \eta$.
Finally, set
\begin{equation*}
\eta=\frac{\ell}{2d} -2\delta_{2}.
\end{equation*}
Since $\ell\ge 6\, d\, \delta_{2}$, we have that $\eta\ge \delta_{2}$
as required and $\deg(g)\le \ell$, as stated. 

On the other hand, the sets in $\cC_{t}$ are realized by sign
conditions given in terms of the $g_{i}$'s.  The sets realized by sign
conditions on $\irr(g)$ have cardinality bounded by those in
$\cC_{t}$.  Since $\eta\ge \frac{\ell}{6d}$, it follows that for each
$ \gamma \in \{\pm 1\}^{\irr(g)}$,
\begin{equation*}
  \card( \mathcal{P}(
  \gamma
  ))  \le \frac{\card( \cP)}{2^{t}} \le  c_{2}
 \frac{\card( \cP)}{ \deg(X) \eta^{d-2}}, 
\end{equation*}
where the last inequality follows from 
Lemma \ref{lem:codim2},
and
$c_{2}$ denotes a suitable constant. This proves the statement in the
case when $X$ is of codimension 2 and the  inequality \eqref{eq:9} are
equalities for all $i$. 

If the inequality \eqref{eq:9} is strict for some
$i$, then there is a polynomial $g_{i} \in I({X(\R)}) \setminus I(X) $
of degree bounded by $\ell_i\le \ell $.  Hence, the hypersurface
$V(g_{i})$ cuts $X$ properly and contains its set of real points. In
particular, $\cP\subset V(g_{i})$. It follows that $g=g_{i}$ has the
appropriate degree and $\cP(\gamma)=\emptyset$ for all $\gamma\in
\{\pm1\}^{\irr(g)}$, which completes the proof for the case when $X$
is of codimension two. 

The cases when the codimension of $X$ is either zero or one  are
simpler and can be proven as in \cite{GK:eddpp, Zahl:ibnpsitd,
  KMSS:udtd}.
\end{proof}
 
A previous version of this paper contained a polynomial partitioning
theorem on varieties of arbitrary dimension. Whereas the proof of this result
contained a gap, we still think that its statement is correct, and we
propose it as a conjecture.

\begin{conjecture} \label{conj:4}
There is a constant $c=c(d)>0$ with the following property. 
Let $X\subset \C^d$ be an irreducible variety of dimension $e$, $\mathcal{P} \subset X(\R)$ a finite subset and $\ell \geq c \,
\delta(X)$.  Then there exists $g \in \R[x_{1},\dots, x_{d}]
_{\leq \ell}$ with $\dim(X\cap V(g)) =\dim(X)-1$ such that, for each $
\gamma \in \{\pm 1\}^{\irr(g)}$,
\begin{equation*}
  \card( \mathcal{P}(
  \gamma
  ))  \le c\, 
  \frac{\card(\mathcal{P})}{\deg(X) \ell^{e}}.  
\end{equation*}
\end{conjecture}

\section{Point-hypersurface incidences} 
\label{sec:proof-theor-refthm:m}

In this section we prove Theorem \ref{thm:main}. To this end, we use
three levels of polynomial partitioning. This leads to a
partition
of the Euclidean space $\R^4$ into semi-algebraic
pieces of various dimensions. We bound separately the number of
incidences contributed by the points of the set $\mathcal{P}$ in each
piece. The contribution from each level of the partitioning is 
essentially the same, up to constant factors, as the claimed bound.

\begin{proof}[Proof of Theorem \ref{thm:main}]
  The procedure performed at each level is similar. For clarity and
  ease of exposition, we prefer to describe each of these level
  separately, even at the expense of repeating some of the arguments.

The \emph{set of incidences} between $\cP$ and $\cH$ is the subset of
$\cP\times \cH$ defined by
\[
\cI(\cP,\cH) = \{ (p,H) \in \cP\times \cH \mid p \in H \}.
\] 
Hence $I(\cP,\cH) = \card(\cI(\cP,\cH))$.  For a subset $\cQ \subset
\cP$, we denote by
 \begin{align*}
 \cI_{<k}(\cQ, \cH) &= \{ (p,H) \in \cI(\cQ, \cH)  \mid  \card(H \cap \cQ) < k \}, \\ 
 \cI_{\geq k}(\cQ, \cH) &= \{ (p,H) \in \cI(\cQ, \cH)   \mid\card(H \cap \cQ) \geq k \}
 \end{align*}
 the set of incidences between $\cQ$ and hypersurfaces of $\cH$
 containing at most $k-1$ points of $\cQ$ and at least $k$ points of
 $\cQ$, respectively. We also set $ I_{<k}(\cQ, \cH) =
 \card(\cI_{<k}(\cQ, \cH))$ and $I_{\geq k}(\cQ, \cH) =
 \card(\cI_{\geq k}(\cQ, \cH))$. Clearly,
\begin{equation}
\label{eq:13}
   I(\cQ,\cH)= I_{< k} (\cQ, \cH)+ I_{\ge k}(\cQ, \cH).
\end{equation}

In the sequel, the dimension $d$ of the ambient space is fixed to $4$.
Hence, all implicit constants in the $O$-notation depend only on the
parameters $k$ and $c$ in the statement of the theorem. 

\subsubsection*{First level partitioning}

Let $D \ge 24$ to be fixed later on. By Theorem \ref{thm:polynomial-partitioning}, there
exists $f \in \R[x_1,x_{2},x_{3},x_4]_{\le D} \setminus \{0\}$ such that, for
each  $\gamma \in \{\pm1\}^{\irr(f)}$,
\begin{equation}
  \label{eqn:firstlevel-right-intermediate1}
\card(\mathcal{P}(\gamma))  = O\Big(  \frac{m}{D^4}\Big),
\end{equation}
where $\cP(\gamma)$ denotes the subset of $\cP$ realized by the signs
$\gamma$ as in \eqref{eqn:sign-pattern}.  Choose a minimal subset
$\Sigma_{1}\subset \{\pm1\}^{\irr(f)}$ realizing all nonempty subsets
of this form.

We partition $\mathcal{P}$ into the disjoint subsets $\mathcal{P}_0 =
\mathcal{P} \cap V(f)$ and $ \mathcal{P}(\gamma) $, $\gamma \in
\Sigma_{1}$.  Set $m_{0}=\card( \cP_{0})$ and $m_\gamma = \card(
\mathcal{P}(\gamma))$ for each $ \gamma$.  Clearly,
\begin{equation*}
m_{0}+ \sum_{\gamma \in \Sigma_{1}} m_{\gamma} = m.
\end{equation*}

We first bound  the number of incidences with hypersurfaces that
contain at least $k$ points in one of the subsets $
\mathcal{P}(\gamma) $.
By the hypothesis \eqref{item:6}, for each
$\gamma \in \Sigma_{1}$ and each subset of $k$ points of
$\mathcal{P}(\gamma)$, there are at most $c$ hypersurfaces in $\cH$
containing these points.  Hence,
\begin{equation}
  \label{eqn:firstlevel-right-intermediate2}
  I_{\ge k}(\mathcal{P}(\gamma),
  \cH) 
  \le c k \binom{m_\gamma}{k} = O(m_{\gamma}^{k}).
\end{equation}
The cardinality of $\Sigma_{1}$ or equivalently, the number of
nonempty subsets of the form $\cP(\gamma)$, is bounded by the number
of connected components of $\R^{4}\setminus V(f)$. By Theorem~\ref{thm:2}, 
applied with $e=0$,
this number is bounded by $O(D^4)$. With
\eqref{eqn:firstlevel-right-intermediate1} and
\eqref{eqn:firstlevel-right-intermediate2}, this implies that
\begin{equation}
\label{eqn:firstlevel-right} 
\sum_{\gamma} I_{\ge k}(\mathcal{P}(\gamma),
\cH) =
O\Big( \sum_{\gamma} \Big(\frac{m}{D^4}\Big)^k \Big)  
= O(m^k D^{4-4k}). 
\end{equation}

We now bound the number of incidences with hypersurfaces that contain
at most $k-1$ points in every $\cP(\gamma)$.  For each $H \in \cH$,
the number of subsets $\cP(\gamma)$ having nonempty intersection with $
H$ is bounded by $\betti(H(\R)\setminus V(f))$.  By Theorem
\ref{thm:2}, this number of connected components is bounded by $
O(D^3)$, because the degree of $H$ is bounded by a constant.  Hence
$\sum_{\gamma} I_{< k} (\mathcal{P}(\gamma), \{H\})\le (k-1) \betti(H
\setminus V(f)) = O(D^{3})$. It follows that
\begin{equation}
\label{eqn:firstlevel-left}
\sum_{\gamma} I_{< k}(\cP(\gamma),\cH) = O(nD^{3}).  
\end{equation}
From \eqref{eq:13}, \eqref{eqn:firstlevel-right} and  \eqref{eqn:firstlevel-left} we
deduce that
\begin{equation}
\label{eqn:firstlevel-intermediate}
{I}(\cP\setminus \cP_{0},\cH) =\sum_{\gamma} I(\cP({\gamma}),\cH) = O(
 n D^3 + m^k D^{4-4k} ).
\end{equation}

We then set
\begin{equation}
 \label{eqn:firstlevel-D}
D = \max\Big(24,  \frac{m^{\alpha_1}}{n^{\beta_1}}\Big) \quad  \text{ with }  \alpha_1 = \frac{k}{4 k
  -1} \text{ and } \beta_1 = \frac{1}{4 k
  -1}.
\end{equation}
If $D=24$, then $m^{\alpha_1} n^{-\beta_1}\le 24$ and so $m^{k}=
O(n)$. In this case, it follows from
\eqref{eqn:firstlevel-intermediate} that ${I}(\cP\setminus
\cP_{0},\cH)= O(n+m^{k})=O(n)$. Otherwise,
\begin{equation*}
  {I}(\cP\setminus \cP_{0},\cH)=   O(  m^{3\alpha_{1}}n ^{1 - 3 \beta_{1}}) =O(m^{1 - \frac{k-1}{4 k -1}} n ^{1 - \frac{3}{4k -1}}).
\end{equation*}
In either case, 
\begin{equation}
\label{eqn:firstlevel}
{I}(\cP\setminus \cP_{0},\cH) =  O(m^{1 - \frac{k-1}{4 k -1}} n ^{1 - \frac{3}{4k -1}}+n).
\end{equation}

\subsubsection*{Second level partitioning}

Let $V(f)= \bigcup_{i\in I} V_{i}$ be the decomposition of the
hypersurface $V(f)$ into irreducible components. Set
$D_{i}=\deg(V_{i})$ for each $i\in I$. Then
\begin{equation}\label{eqn:secondlevel-bezout}
  \sum_{i\in I}D_{i}=\deg(V(f)) \le D.
\end{equation}
We choose a partition of the finite set $\mathcal{P}_0=\cP\cap
V(f)$ into disjoint subsets $\mathcal{Q}_{i}$,
$i \in I$, by assigning each point in $\mathcal{P}_{0}$ to one of the
subsets $\mathcal{Q}_{i}$ corresponding to an irreducible component
$V_{i}$ it belongs to.  Set $l_{i} = \card( \mathcal{Q}_{i})$ for
each $i\in I$. Then 
\begin{equation}
  \label{eqn:secondlevel-aggregrate}
\sum_{i}l_{i}=m_{0}.  
\end{equation}

Fix $i\in I$ and let $E_i \geq 24 D_i$. By
Theorem \ref{thm:polynomial-partitioning}, there exists $g_{i} \in
\R[x_1,x_{2},x_{3},x_4]_{\le E_{i}}$ such that $\dim(V_{i}\cap
V(g_{i}))= 2$ and, for each $\delta \in \{\pm1\}^{\irr(g_{i})}$,
\begin{equation}
  \label{eqn:secondlevel-right-intermediate1}
\card( \mathcal{Q}_{i}(\delta))  = O\Big(  \frac{l_{i}}{D_{i}E_{i}^3}\Big).
\end{equation}
Choose a minimal subset $\Sigma_{2,i}\subset \{\pm1\}^{\irr(g_{i})}$
realizing all nonempty subsets of the form $\cQ_{i}(\delta)$.

Consider the surface $W_i=V_{i}\cap V(g_{i})=V(f_i,g_i)$ and partition
$\cQ_{i}$ into the disjoint subsets $\cQ_{i,0}=\cQ_{i}\cap W_{i}$ and
$\cQ_{i}(\delta)$, $\delta \in \Sigma_{2,i}$.  We set $l_{i,0}=\card(
\cQ_{i,0})$ and $l_{i,\delta}=\card( \cQ_{i}(\delta))$ for each
$\delta$. Clearly,
\begin{equation*}
 l_{i,0}+\sum_{\delta \in \Sigma_{2,i}}
l_{i,\delta}= l_{i} \quad \text{ and } \quad  
\sum_{i}l_{i,0}= \card\Big(\cP\cap \bigcup_{i}W_{i}\Big).
\end{equation*}

We follow the same approach as in the previous case, and we first
bound the number of incidences with hypersurfaces that contain at
least $k$ points in some $\cQ_{i}(\delta)$.  Similarly as in
\eqref{eqn:firstlevel-right-intermediate2}, the hypothesis
\eqref{item:6} implies that, for each $\delta$,
\begin{equation}
  \label{eqn:secondlevel-right-intermediate2}
  I_{\ge k}(\cQ_{i}(\delta),\cH) \le ck {l_{i,\delta} \choose k} = O(l_{i,\delta}^{k}).
\end{equation}
The cardinality of $\Sigma_{2,i}$ is bounded
by $\betti(V_{i}(\R) \setminus V(g_{i}))$
which, by Theorem \ref{thm:2}, is bounded by
$O(D_{i}E_{i}^{3})$. 
With \eqref{eqn:secondlevel-right-intermediate1} and \eqref{eqn:secondlevel-right-intermediate2}, this implies that
\begin{equation}
  \label{eqn:secondlevel-right}
\sum_{\delta}  I_{\ge k} (\cQ_{i}(\delta),\cH)
=O\Big( \sum_{\delta}   \Big(\frac{l_{i}}{D_{i}E_{i}^3}\Big)^{k}\Big)
= O( l_{i}^k D_i ^{1-k}E_i^{3-3k} ).
\end{equation}

We now bound the number of incidences with hypersurfaces that contain
at most $k-1$ points in every $\cQ_{i}(\delta)$.  Let $H\in \cH$ and,
for the moment, suppose that $V_{i}\not\subset H$.  
Since $V_{i}$ is
an algebraic variety over $\C$ with $\dim(V_{i})=3$, by \cite[Chapter
I, Proposition 7.1]{Hartshorne:ag} we have that either $H\cap V_{i}$
is empty or of dimension 2. Moreover, the degree of $H$ is bounded by
a constant.  The number of subsets of the form $\cQ_{i}(\delta)$ with
nonempty intersection with $H$ is bounded by $\betti((H\cap V_{i})(\R)
\setminus V(g_{i}))$.  By Theorem \ref{thm:2}, this number is bounded
by $O(D_i E_i^2)$.  
If we note by $\cH_{i}$ the set of hypersurfaces
of $\cH$ not containing $V_{i}$, then
\begin{equation*}
\sum_{\delta} I_{< k}(\mathcal{Q}_{i}(\delta), \cH_{i})= O(nD_{i}E_{i}^{2}).
\end{equation*}
On the other hand, by the hypothesis \eqref{item:5}, there are at most
$3$ hypersurfaces $H\in \cH$ containing $V_{i}$, and each of them
contains the $l_{i}$ points of $\cQ_{i}$. Hence
\begin{equation}
  \label{eqn:secondlevel-left-intermediate3}
I_{< k}(\cQ_{i}\setminus \cQ_{i,0}, \cH\setminus \cH_{i})\le I(\cQ_{i}, \cH\setminus \cH_{i})\le  3 l_{i}. 
\end{equation}
By  \eqref{eqn:secondlevel-right} and
\eqref{eqn:secondlevel-left-intermediate3}, 
\begin{equation}
\label{eqn:secondlevel-left-intermediate4}
  I(\cQ_{i}\setminus \cQ_{i,0}, \cH) = \sum_{\delta} I(\cQ_{i}(\delta),
  \cH) = O(n D_i E_i^2 + l_i^k D_i^{1-k } E_i^{3-3k}+l_{i}).
\end{equation}

We set 
\begin{equation}
  \label{eqn:secondlevel-left-intermediate6}
 E_{i}= \max\Big( 24 D_i,  \Big(\frac{l_i}{D_{i}}\Big)^{\alpha_2}
 \frac{1}{n^{\beta_2}}\Big) \quad \text{ with } \alpha_2 =\frac{k}{3k -1} \text{ and }  \beta_2 =
\frac{1}{3 k -1}.
\end{equation}

If $E_i = 24 D_i$, then $ (\frac{l_i}{D_{i}})^{\alpha_2}
{n^{-\beta_2}}\le 24 D_{i}$. In this case, the first term in the
right-hand side of \eqref{eqn:secondlevel-left-intermediate4} controls
the second one. Otherwise, both terms are equal up to a constant
factor. We deduce from \eqref{eqn:secondlevel-left-intermediate4} that
\begin{equation}
\label{eqn:secondlevel-intermediate1}
  I(\cQ_{i}\setminus \cQ_{i,0}, \cH) = 
  \begin{cases}
O(n D_{i}^{3} +l_{i}) & \text{ if } E_i = 24 D_i,\\
    O( n^{1 - 2 \beta_2}  l_i^{2
     \alpha_2} D_i^{1 - 2 \alpha_2} +l_{i}) &
   \text{ otherwise}.
  \end{cases}
\end{equation}
By \eqref{eqn:secondlevel-bezout},
\begin{equation}
  \label{eqn:secondlevel-intermediate2}
  \sum_{i }  n D_{i}^{3}\le nD^{3} = O(m^{1 - \frac{k-1}{4 k
      -1}} n ^{1 - \frac{3}{4k -1}}+n),
\end{equation}
as the term $nD^{3}$ appears in \eqref{eqn:firstlevel-intermediate}
and is accounted for in \eqref{eqn:firstlevel}.  Using the H\"older
inequality as well as
\eqref{eqn:secondlevel-bezout} and \eqref{eqn:secondlevel-aggregrate},
we get
\begin{multline}
  \label{eqn:secondlevel-intermediate3}
  \sum_{i}  n^{1 - 2 \beta_2} l_i^{2
     \alpha_2} D_i^{1 - 2 \alpha_2}  \le 
 n^{1 - 2 \beta_2} \Big( \sum_{i} l_i \Big)^{2
     \alpha_2}   \Big(\sum_{i} D_i \Big)^{1 - 2 \alpha_2} \\ \le 
 n^{1 - 2 \beta_2} m_{0}^{2
     \alpha_2}  D^{1 - 2 \alpha_2} .
\end{multline}

We now substitute
the value of $D$ from \eqref{eqn:firstlevel-D} and those of
$\alpha_1,\alpha_2,\beta_1$ and $\beta_2$ in the above expression. If
$D = 24$, then $m ^{k}= O(n)$ and so $ n^{1 - 2 \beta_2} m_{0}^{2
  \alpha_2} D^{1 - 2 \alpha_2} = n^{1 - 2 \beta_2} m_{0}^{2 \alpha_2}
= O(n)$.  Otherwise,
\begin{equation}
  \label{eqn:secondlevel-intermediate4}
  n^{1 - 2 \beta_2} m_{0}^{2
    \alpha_2}  D^{1 - 2 \alpha_2} \le   n^{1 - 2 \beta_2} m^{2
    \alpha_2}  (m^{\alpha_1} n^{-\beta_1})^{1 -
    2 \alpha_2} = m^{1 - \frac{k-1}{4 k
      -1}} n ^{1 - \frac{3}{4k -1}}.
\end{equation}
It follows from \eqref{eqn:secondlevel-intermediate1}, \eqref{eqn:secondlevel-intermediate2},
\eqref{eqn:secondlevel-intermediate3},
\eqref{eqn:secondlevel-intermediate4} and \eqref{eqn:secondlevel-aggregrate}
that
\begin{align}
I\Big(\cP_{0}\setminus \bigcup_{i}\cQ_{i,0}, \cH\Big)& = 
  \sum_{i}    I(\cQ_{i}\setminus \cQ_{i,0}, \cH)  \nonumber\\
& = O\Big(
  \sum_{i }  n D_{i}^{3} +  \sum_{i} n^{1 - 2 \beta_2}  l_i^{2
     \alpha_2} D_i^{1 - 2 \alpha_2} + \sum_{i}
   l_{i}\Big) \nonumber \\ & = O(
m^{1 - \frac{k-1}{4k-1}} n^{1 - \frac{3}{4 k -1}} +
n+m_{0}).   \label{eqn:secondlevel}
\end{align}

\subsubsection*{Third level partitioning}
For each $i \in I$, let $W_i = \bigcup_{j \in J_i} W_{i,j}$ be the
decomposition of the surface $W_i=V(f_i,g_i)$ into
irreducible components. Set $\Delta_{i,j}=\deg(W_{i,j})$ for each
$j$. By B\'ezout's inequality, 
 \begin{equation}
\label{eq:10}
   \sum_{j\in J_{i}}\Delta_{i,j}= \deg(W_{i})\le D_{i}E_{i}.
 \end{equation}
 We denote by $W_i(\R)_{0}$ and $W_{i,j}(\R)_{0}$ the set of isolated
 points of the semi-algebraic sets $W_i(\R)$ and $W_{i,j}(\R)$,
 respectively.  We then choose an arbitrary partition of the set
 $\cQ_{i,0}= \cQ_{i} \cap W_{i}$ into disjoints subsets $\cR_{i,j}$,
 $j\in J_{i}$, such that
 \begin{equation*}
\cR_{i,j}\subset W_{i,j}(\R) \quad \text{ and } \quad \cR_{i,j} \cap
 W_{i,j}(\R)_{0} \subset W_i(\R)_{0}.
\end{equation*}
 Set
 $e_{i,j}=\card(\cR_{i,j})$ for each $j$. Then 
 \begin{equation}
  \label{eqn:thirdlevel-aggregrate1}
   \sum_{j}e_{i,j}=
   l_{i,0}.
 \end{equation}

 Let $j \in J_i$. Being an irreducible component of
 $W_{i}=V(f_{i},g_{i})$, the variety $W_{i,j}$ is partially defined at
 degree $E_{i}$. Let
$F_{i,j} \ge 24
E_i$, 
to be fixed later on. By Theorem
 \ref{thm:polynomial-partitioning}, there exists $h_{i,j} \in
 \R[x_{1},x_{2},x_{3},x_{4}]_{\le F_{i,j}}$ such that
 $\dim(W_{i,j}\cap V(h_{i,j}))=1$ and, for each $\eta \in
 \{\pm1\}^{\irr(h_{i,j})}$,
\begin{equation}
\label{eqn:thirdlevel-1-right-intermediate1}
\card( \mathcal{R}_{i,j}( \eta))  = O\Big( \frac{ e_{i,j}}{\Delta_{i,j}
  F_{i,j}^2}\Big).
\end{equation}
Similarly as before, choose a minimal subset $\Sigma_{3,i,j}\subset
\{\pm1\}^{\irr(h_{i,j})}$ realizing all nonempty subsets of the form
$\cR_{i,j}(\eta)$.

Consider the curve $Y_{i,j}=W_{i,j}\cap V(h_{i,j})$ and partition
$\cR_{i,j}$ into the disjoint subsets $\cR_{i,j,0}=\cR_{i,j}\cap
Y_{i,j}$ and
$ \cR_{i,j}(\eta)$, 
$\eta \in
\Sigma_{3,i,j}$.  Set also $e_{i,j,0}=\card(\cR_{i,j,0})$ and
$e_{i,j,\eta}=\card( \cR_{i,j}(\eta))$ for each $\eta$. Hence,
\begin{equation*}
e_{i,j,0}+\sum_{\eta\in \Sigma_{3,i,j}}e_{i,j,\eta}= e_{i,j}. 
\end{equation*}

We first bound the number of incidences of $\cR_{i,j}\setminus
\cR_{i,j,0}$ with hypersurfaces that contain at least $k$ points in
some $\cR_{i,j}(\eta)$.  Similarly as for
\eqref{eqn:firstlevel-right-intermediate2} and
\eqref{eqn:secondlevel-right-intermediate2}, the hypothesis
\eqref{item:6} implies that, for each $\eta$,
\begin{equation}
  \label{eqn:thirdlevel-1-right-intermediate2}
  I_{\ge k}(\cR_{i,j}(\eta),\cH) \le ck {e_{i,j,\eta}\choose k} = 
  O(e_{i,j,\eta}^{k}). 
\end{equation} By Proposition \ref{prop:1}, 
after dehomogenizing,
there are coprime polynomials
 $\widetilde{f}_{i,j},\widetilde{g}_{i,j} \in
 \R[x_{1},x_{2},x_{3},x_{4}]$ such that $W_{i,j}$ is an irreducible
 component of the variety $ \widetilde{W}_{i,j} =
 V(\widetilde{f}_{i,j},\widetilde{g}_{i,j})$ and
\begin{equation}
 \label{eq:12}
 \deg(\widetilde{f}_{i,j})   \deg(\widetilde{g}_{i,j}) = O(\Delta_{i,j}).
\end{equation}
Since $W_{i,j}$ is partially defined at degree
$E_{i}$, we can furthermore
deduce 
  that 
$ \deg(\wt f_{i,j}), \deg(\wt g_{i,j})\le
E_{i}$.

The number of nonempty subsets of the form $\cR_{i,j}(\eta)$ is
bounded by the number of connected components of
$\wt{W}_{i,j}(\R)\setminus V(h_{i,j})$, as explained in Remark
\ref{rem:1}. By Theorem \ref{thm:2} and \eqref{eq:12}, this number of
connected components is bounded by
\begin{equation}
\label{eqn:thirdlevel-1-right-intermediate3}
\betti(\wt{W}_{i,j}(\R)\setminus
V(h_{i,j}))
= 
O(\deg(\widetilde{f}_{i,j})
\deg(\widetilde{g}_{i,j}) F_{i,j}^{2})= O(\Delta_{i,j}F_{i,j}^{2}).
\end{equation}
By \eqref{eqn:thirdlevel-1-right-intermediate1},
\eqref{eqn:thirdlevel-1-right-intermediate2} and
\eqref{eqn:thirdlevel-1-right-intermediate3},
\begin{equation}
  \label{eqn:thirdlevel-1-right}
 \sum_{\eta}  I_{\ge
  k}(\cR_{i,j}(\eta),\cH) =  O\Big(  \sum_{\eta}
  \Big( \frac{e_{i,j}}{\Delta_{i,j}F_{i,j}^{2}}\Big)^{k} \Big) = O( e_{i,j}^k \Delta_{i,j} ^{1-k}F_{i,j}^{2-2k} ).
\end{equation}

We now bound the number of incidences of $\cR_{i,j}\setminus
\cR_{i,j,0}$ with hypersurfaces that contain at most $k-1$ points in
every $\cR_{i,j}(\eta)$. 
We would like to use an argument similar to those used above in the
case of first and second level partitioning, and bound, for each $H \in
\mathcal{H}$, the number of these incidences on $H$, by $k$ times the
number of connected components of $\R^{4}\setminus V(h_{i,j})$ having
nonempty intersection with $H\cap W_{i,j}(\R)$ using Theorem
\ref{thm:2}. However, there are two difficulties in this
approach. First, unlike  the prior cases,  we do not
have the equations defining $W_{i,j}$, but rather those of a possibly
larger variety $\widetilde{W}_{i,j}$. This is not a serious problem,
since clearly the number of connected components of $\R^{4}\setminus
V(h_{i,j})$, met by the possibly larger set $H\cap
\widetilde{W}_{i,j}(\R)$ is an upper bound on the number of connected
components of $\R^{4}\setminus V(h_{i,j})$ having nonempty
intersection with $H\cap W_{i,j}(\R)$. The second difficulty is more
serious. To apply Theorem \ref{thm:2} to obtain a
sufficiently good upper bound (see
\eqref{eqn:thirdlevel-1-left-intermediate1} below) we require that the
dimension of the intersection $H\cap \widetilde{W}_{i,j}$ is one (if
$H\cap \widetilde{W}_{i,j}$ is non-empty), and this requirement might
not be satisfied. To circumvent this difficulty, for each $H
\in \mathcal{H}$, we partition $W_{i,j}(\R)$ into two semi-algebraic
subsets, namely $G_{i,j}(H)$ and $B_{i,j}(H)$. The non-isolated points
of $W_{i,j}(\R)$ which belong $B_{i,j}(H)$ are points having an open
neighborhood in $W_{i,j}(\R)$ (with respect to its Euclidean topology)
which is contained also in $H$.  Such points have the bad property
that the intersection of $W_{i,j}(\R)$ with a small perturbation of
$H$ could be empty in a neighborhood of such a point.  We bound
incidences created by points in the various $B_{i,j}(H)$ using a
separate argument (see inequality \eqref{eqn:thirdlevel-01} below).

On the other hand,  to bound the incidences in $H \cap G_{i,j}(H)$, 
we show (using Proposition \ref{prop:A})  that it is possible to replace $H$ by a slightly perturbed hypersurface $\widetilde{H} \subset \R^4$ of degree
twice the degree of $H$, satisfying:
\begin{enumerate}
\item
every connected component of $\R^{4}\setminus V(h_{i,j})$ having nonempty
intersection with $H\cap G_{i,j}(H)$ also has a non-empty intersection with
$\widetilde{H} \cap \widetilde{W}_{i,j}(\R)$;
\item
the dimension of $\widetilde{H} \cap \widetilde{W}_{i,j}$ is equal to $1$ if  $\widetilde{H} \cap \widetilde{W}_{i,j} \neq \emptyset$.
\end{enumerate}

This, allows us to obtain the necessary estimate 
(see \eqref{eqn:thirdlevel-1-left-intermediate1} below)
on the number of connected components of 
$\R^{4}\setminus V(h_{i,j})$ having nonempty
intersection with $H\cap G_{i,j}(H)$ using Theorem \ref{thm:2}.
We now make the above arguments precise as follows.

Given $H\in \cH$, we denote by
$B_{i,j}(H)\subset W_{i,j}(\R)$ the semi-algebraic subset of points
$p\in W_{i,j}(\R)$ having an open neighborhood, in the Euclidean
topology of $W_{i,j}(\R)$,  contained in $H$. We also set
$G_{i,j}(H) = W_{i,j}(\R) \setminus B_{i,j}(H)$.  Notice that
unlike $B_{i,j}(H)$, the semi-algebraic set $G_{i,j}(H)$ is not
necessarily contained in $H$, and that $W_{i,j}(\R)_{0} \cap H \subset
B_{i,j}(H)$.

For any finite subset $\cR\subset W_{i,j}(\R)$ we set 
\begin{displaymath}
 \cI^{\bad}(\cR,\mathcal{H}) =  
 \bigcup_{H \in \mathcal{H}}
 \cI
 (\cR \cap
 B_{i,j}(H),\mathcal{H})  \quad \text{ and } \quad 
 \cI^{\good}(\cR,\mathcal{H}) = 
 \bigcup_{H \in \mathcal{H}}
 \cI
 (\cR \cap G_{i,j}(H),\mathcal{H}).
\end{displaymath}

We also set $I^{\bad}(\cR,\mathcal{H})=\card(
\cI^{\bad}(\cR,\mathcal{H}) )$ and
$I^{\good}(\cR,\mathcal{H})=\card(
\cI^{\good}(\cR,\mathcal{H}) )$. Clearly, 
\begin{equation*}
  I(\cR,\mathcal{H})= I^{\bad}(\cR,\mathcal{H})+I^{\good}(\cR,\mathcal{H}).
\end{equation*}

We first treat the incidences in $G_{i,j}(H)$. Write $H=V(b)$ with $b
\in \R[x_{1},x_{2},x_{3},x_{4}]$. The number of nonempty subsets of
the form $\cR_{i,j}(\eta) \cap G_{i,j}(H)$ is bounded by the number of
connected components of $\R^{4}\setminus V(h_{i,j})$ having nonempty
intersection with $H\cap G_{i,j}(H)$. By Proposition \ref{prop:A},
this number of connected components is bounded by the 
number of connected components of the semi-algebraic set
\begin{equation}
  \label{eq:11}
  (  \wt W_{i,j}\cap V(b^{2}-\varepsilon))(\R)\setminus V(h_{i,j})=V(
  b^{2}-\varepsilon, \wt
  f_{i,j}, \wt g_{i,j})(\R)\setminus V(h_{i,j})
\end{equation}
for any $\varepsilon>0$ sufficiently small. Choosing a possibly
smaller $\varepsilon>0$, we also have that 
$ \dim(\wt W_{i,j}\cap V(b^{2}-\varepsilon))=
1$
if $\wt W_{i,j}\cap V(b^{2}-\varepsilon) \neq \emptyset$.
To see this, observe that the set of critical values of $b^2$
restricted to $\mathrm{reg}(\wt W_{i,j})$ is finite using Sard's
theorem \cite[page 255]{Loj}, and hence, for all
$\varepsilon>0$  small enough, $\varepsilon$ is a regular value of
$b^2$ restricted to $\mathrm{reg}(\wt W_{i,j})$. It follows that
either $\wt W_{i,j}\cap V(b^{2}-\varepsilon)$ is empty, or $ \dim(\wt
W_{i,j}\cap V(b^{2}-\varepsilon)) =1$, using the implicit function
theorem (see for example \cite[page 19]{Griffiths-Harris}).

For any such 
choice of $\varepsilon$, by
Theorem \ref{thm:2} and \eqref{eq:12}, the number of connected
components of the semi-algebraic set in \eqref{eq:11} is bounded by
\begin{equation}
\label{eqn:thirdlevel-1-left-intermediate1}
 O(  \deg(b^{2}-\varepsilon)\deg(\wt
  f_{i,j})\deg(\wt g_{i,j}) \deg(h_{i,j}) )= 
  O(\Delta_{i,j} F_{i,j}).
 \end{equation}
Thus
\begin{equation}
  \label{eqn:thirdlevel-1-left}
  \sum_{\eta}I_{<k}^\good(\cR_{i,j} (\eta),\cH) = O(n \Delta_{i,j} F_{i,j}),
\end{equation}
where 
\[
I_{<k}^\good(\cR_{i,j} (\eta),\cH) = \card(\cI^\good(\cR_{i,j} (\eta),\cH) \cap \cI_{<k}(\cR_{i,j} (\eta),\cH)).
\]

Gathering together \eqref{eqn:thirdlevel-1-right} and \eqref{eqn:thirdlevel-1-left}, 
we obtain that
\begin{multline}
\label{eqn:thirdlevel-intermediate1}
 I^{\good}(\cR_{i,j}\setminus \cR_{i,j,0}, \cH)=\sum_{\eta} I^{\good}(\cR_{i,j}(\eta),
\cH) \\ = O\big( n
\Delta_{i,j} F_{i,j} +e_{i,j}^k \Delta_{i,j}^{1-k}F_{i,j}^{2-2k} \big).
\end{multline}

We  set
\begin{equation*}
F_{i,j}= \max\Big(24 E_i,
\Big(\frac{e_{i,j}}{\Delta_{i,j}}\Big)^{\alpha_3} \frac{1}{n^{\beta_3}}\Big)
\quad \text{ with }  \alpha_3 = \frac{k}{2k -1} \text{ and } \beta_3 =\frac{1}{2 k -1}.
\end{equation*}

If $F_{i,j} = 24 E_i $, then
$(\frac{e_{i,j}}{\Delta_{i,j}})^{\alpha_3} {n^{-\beta_3}}=
O(E_{i})$. In this case, the first term in the right-hand side of
\eqref{eqn:thirdlevel-intermediate1} controls the second one and,
otherwise, both terms are equal up to a constant factor. Hence, 
\begin{equation}
\label{eqn:thirdlevel-intermediate3}
 I^\good(\cR_{i,j}\setminus \cR_{i,j,0}, \cH)=
  \begin{cases}
    O(n
    \Delta_{i,j}E_i)& \text{ if } F_{i,j} = 24 E_i,\\
    O(n^{1-\beta_{3}} e_{i,j}^{\alpha_{3}} \Delta_{i,j}^{1-\alpha_{3}})& \text{ otherwise}.
  \end{cases}
\end{equation}
By 
\eqref{eq:10}
and B\'ezout's inequality,
\begin{equation}
  \label{eqn:thirdlevel-intermediate4}
  \sum_{i,j}n
  \Delta_{i,j}E_i \le  \sum_{i} nD_{i}E_{i}^{2} = O(m^{1 - \frac{k-1}{4 k
      -1}} n ^{1 - \frac{3}{4k -1}}+n),
\end{equation}
as shown when passing from \eqref{eqn:secondlevel-left-intermediate4}
to \eqref{eqn:secondlevel}.
Else, applying the  H\"older inequality together with  \eqref{eqn:thirdlevel-aggregrate1} and 
\eqref{eqn:secondlevel-bezout},
\begin{multline}
\label{eqn:thirdlevel-intermediate5}
\sum_{i,j}  n^{1-\beta_{3}} e_{i,j}^{\alpha_{3}}
\Delta_{i,j}^{1-\alpha_{3}} \le n^{1-\beta_{3}}\Big( \sum_{i,j}   e_{i,j}\Big)^{\alpha_{3}}
\Big( \sum_{i,j} \Delta_{i,j}\Big)^{1-\alpha_{3}} \\ \le 
 n^{1-\beta_{3}} m^{\alpha_{3}}
\Big( \sum_{i} D_{i}E_{i}\Big)^{1-\alpha_{3}}.
\end{multline}

Recall that $E_{i}= \max\big(24D_i,
\big(\frac{l_i}{D_{i}}\big)^{\alpha_2} {n^{-\beta_2}}\big)$ as in
\eqref{eqn:secondlevel-left-intermediate6}. Hence
\begin{align}  
\sum_{i} D_{i}E_{i} &= O\Big( \sum_i D_i^2 + {n^{-\beta_2}} \sum_{i} 
{l_i}^{\alpha_2}{D_{i}}^{1-\alpha_2} \Big) \nonumber \\ 
&= O\Big( \sum_i D_i^2 + {n^{-\beta_2}} \Big(\sum_{i}
{l_i}\Big)^{\alpha_2} \Big(\sum_{i}{D_{i}}\Big)^{1-\alpha_2} \Big)
\nonumber \\ 
& =O\big( D^{2}+{n^{-\beta_2}} m^{\alpha_2} D^{1-\alpha_2} \big) . \label{eqn:thirdlevel-intermediate6}
\end{align}
Recall also that $D = \max\big(24, {m^{\alpha_1}}{n^{-\beta_1}}\big)$
as in \eqref{eqn:firstlevel-D}.  If $D=24$, then  $m^{k}=O(n)$.  In this case,
$\sum_i D_i E_i = O(1)$. Otherwise, substituting $D =m^{\alpha_1}
n^{-\beta_1}$ in \eqref{eqn:thirdlevel-intermediate6} and the sum
$\sum_{i}D_{i}E_{i}$ into \eqref{eqn:thirdlevel-intermediate5},
\begin{align}
   n^{1-\beta_{3}} m^{\alpha_{3}}
\Big( \sum_{i} D_{i}E_{i}\Big)^{1-\alpha_{3}} 
& =   O\big( n^{1-\beta_{3}} m^{\alpha_{3}} ({n^{-\beta_2}} m^{\alpha_2}
(m^{\alpha_1}  n^{-\beta_1})^{1-\alpha_2})^{1-\alpha_{3}}\big)
\nonumber \\
& =O\big( m^{1 - \frac{k-1}{4 k -1}} n^{1 - \frac{3}{4 k-1}} \big).
  \label{eqn:thirdlevel-intermediate8}
\end{align}

It follows from \eqref{eqn:thirdlevel-intermediate3},
\eqref{eqn:thirdlevel-intermediate4},
\eqref{eqn:thirdlevel-intermediate5},
\eqref{eqn:thirdlevel-intermediate6}
and\eqref{eqn:thirdlevel-intermediate8}, that
\begin{align}
\nonumber
I^{\good}\Big(\bigcup_{i}\cQ_{i,0} \setminus \bigcup_{i,j}\cR_{i,j,0},\cH \Big) & =
\sum_{i,j} I^\good(\cR_{i,j}\setminus \cR_{i,j,0},\cH)  \\
&= O\Big( 
\sum_{i,j}
\Big( n
\Delta_{i,j} F_{i,j} +e_{i,j}^k \Delta_{i,j}^{1-k}F_{i,j}^{2-2k}  \Big)
\Big) \label{eq:14}
\\
&=O\big(
m^{1 - \frac{k-1}{4 k -1}} n^{1 - \frac{3}{4 k-1}}+n\big). \label{eqn:thirdlevel}
\end{align}

Finally, we treat the incidences in $B_{i,j}(H)$.  We claim that for
each $p \in \cR_{i,j} \setminus W_{i}(\R)_{0}$ there are at most 3
hypersurfaces in $\cH$ such that $p\in B_{i,j}(H)$.  To see this,
observe that $p \in B_{i,j}(H)$ implies that $H$ contains an open
neighborhood $U\subset W_{i,j}(\R)$ of $p$. Since $p$ is not an
isolated point of $W_{i,j}(\R)$, if follows that $U$ is of real
dimension at least $1$. The claim then follows from the hypothesis
\eqref{item:5}.

Hence,
\begin{equation}
\label{eqn:thirdlevel-01}
I^\bad(\cR_{i,j} \setminus W_{i}(\R)_{0},\mathcal{H}) \leq 3 e_{i,j}.
\end{equation}

The incidences of $\cQ_{i}$ with hypersurfaces $H\in\cH$ containing
$V_{i}$ are already accounted for in
\eqref{eqn:secondlevel-left-intermediate3}. Hence, we can suppose that
$V_{i}$ is not contained in $H$.  In this case, by Theorem
\ref{thm:2}, $\card(H \cap W_i(\R)_{0}) \le \betti(V(b,f_{i},g_{i}))=
O(D_i E_i^2)$, where $b$ is the polynomial defining $H$.  Together with
\eqref{eqn:secondlevel-left-intermediate3}, this implies that
\begin{equation}
\label{eqn:thirdlevel-00}
\sum_{j} I^\bad(\cR_{i,j} \cap W_{i}(\R)_{0},\mathcal{H}) = O(n D_i
E_i^2 +l_{i})
\end{equation}
It follows from  \eqref{eqn:thirdlevel-01}, \eqref{eqn:thirdlevel-aggregrate1} and
\eqref{eqn:thirdlevel-00} that
 \begin{equation*}
\sum_{j} I^\bad(\cR_{i,j},\mathcal{H}) = 
O(n D_i E_i^2 + l_{i} ) + O(l_{i,0}) =
O(n D_i E_i^2 + l_{i} ).
\end{equation*}
This bound already appears in
\eqref{eqn:secondlevel-left-intermediate4}. The contribution of the
sum of these terms over $i\in I$ is accounted for in \eqref{eqn:secondlevel} and can be absorbed into the bound
\eqref{eqn:thirdlevel}, after adding the term $m$.
We conclude that
\begin{equation}
  \label{eq:35}
I\Big(\bigcup_{i}\cQ_{i,0} \setminus
\bigcup_{i,j}\cR_{i,j,0},\cH \Big)  =O\big(
m^{1 - \frac{k-1}{4 k -1}} n^{1 - \frac{3}{4 k-1}}+m+n\big).
\end{equation}

\subsubsection*{The case of curves and conclusion of the proof}
Finally, we bound the number of incidences that occur on the curves
$Y_{i,j}=W_{i,j}\cap V(h_{i,j})$. 

For each $i,j$, set $\cR_{i,j,0}=\cR_{i,j}\cap Y_{i,j}$. Let
$Y_{i,j}=\bigcup_{l\in L_{i,j}}Y_{i,j,l}$ be the decomposition of
$Y_{i,j}$ into irreducible components and consider an arbitrary partition
of $\cR_{i,j,0}$ into disjoint subsets $\cS_{i,j,l}$, $l\in L_{i,j}$,
with $\cS_{i,j,l}\subset Y_{i,j,l}$ for all $l$.

Let $l\in L_{i,j}$ and $H\in \cH$. If $Y_{i,j,l}$ is not contained in
$H$, then the number of incidences between $\cS_{i,j,l}$ and this
hypersurface is bounded by $\card(Y_{i,j,l}\cap H)$. From B\'ezout's
inequality, we deduce that
\begin{equation}
  \label{eqn:fourthlevel-intermediate1}
  I(\cS_{i,j,l},\{H\})\le 
  \begin{cases}
 \deg(H)\deg(Y_{i,j,l}) & \text{ if } Y_{i,j,l}\not\subset H, \\
\card (\cS_{i,j,l}) & \text{ if } Y_{i,j,l}\subset H.
  \end{cases}
\end{equation}
The hypothesis \eqref{item:5} implies that, for each $l$, there are at most $3$ hypersurfaces
in $\cH$ containing $Y_{i,j,l}$. It follows from
\eqref{eqn:fourthlevel-intermediate1} that
\begin{align}
  I(\cR_{i,j,0}, \cH)& = \sum_{l\in L_{i,j}}\sum_{H\in \cH}
  I(\cS_{i,j,l},\{H\}) \nonumber \\ & = O\Big( \sum_{l,H} \deg(Y_{i,j,l})\Big)
  + 3 \sum_{l} \card (\cS_{i,j,l}) \nonumber \\ &= O( n \deg(Y_{i,j}) +
  \card (\cR_{i,j,0})) .   \label{eq:38}
\end{align}
By B\'ezout's inequality,  $\deg(Y_{i,j}) \le \Delta_{i,j}F_{i,j}$.
Using \eqref{eq:38},
\begin{equation}
  \label{eq:fourthlevel:intermediate3}
  I\Big(\cP \cap  \bigcup_{i,j} Y_{i,j},\cH\Big) = \sum_{i,j}
  I(\cR_{i,j,0}, \cH) = O\Big( \sum_{i,j}n \Delta_{i,j}F_{i,j} + \sum_{i,j}\card(\cR_{i,j,0})\Big).
\end{equation}
The first sum in the right-hand side of
\eqref{eq:fourthlevel:intermediate3} appears in \eqref{eq:14}
and is already accounted for in \eqref{eqn:thirdlevel}.  By
construction, the family of sets $\{\cR_{i,j,0}\}_{i,j}$ is a
partition of $\cP \cap \bigcup_{i,j} Y_{i,j}$. Therefore, the sum of
their cardinalities is bounded by $m$. Hence,
\begin{equation}
\label{eqn:fourthlevel}
  I\Big(\cP \cap  \bigcup_{i,j} Y_{i,j},\cH\Big)  = O\big( m^{1 - \frac{k-1}{4 k -1}} n^{1 - \frac{3}{4 k-1}}+m+n \big).
\end{equation}

The statement now follows by summing up the contributions from
\eqref{eqn:firstlevel}, \eqref{eqn:secondlevel}, \eqref{eq:35} and
\eqref{eqn:fourthlevel}.
\end{proof}


We close this paper by proposing the next conjecture  on the number of
point-hypersurfaces incidences in higher dimension. 

\begin{conjecture}
\label{conj:main-conjecture}
Let $d,k,c \geq 1$, and let $\mathcal{P}$ be a finite set of points of
$\R^{d}$ and $\cH$ a finite set of hypersurfaces of $\C^d$ satisfying
the following conditions:
\begin{enumerate}
\item \label{item:7} the degrees of the hypersurfaces in $\cH$ are bounded
  by $c$; 
\item \label{item:8} the intersection of any family of $d$ distinct
  hypersurfaces in $\cH$ is finite;
\item \label{item:9} for any subset of $k$ distinct points in $\cP$, the
  number of hypersurfaces in $\cH$ containing them is bounded by $c$.
\end{enumerate}
Set  $m=\card(\mathcal{P})$ and $n=\card(\cH)$. Then
\[
I(\mathcal{P},\cH) = O_{d,k,c}(m^{1-\frac{k-1}{d k-1}} n^{1 - \frac{d-1}{d k-1}} +m +n).
\]
\end{conjecture}

This conjecture is suggested by the bound that follows from the first
level of the polynomial partitioning method applied to this problem.
It contains the statements of the Szemer\'edi-Trotter theorem
\ref{thm:szemeredi-trotter}, the results of Zahl and Kaplan,
Matou{\v{s}}ek, Sharir and Safernov\'a in three dimensions
\cite{Zahl:ibnpsitd,KMSS:udtd}, and Theorem \ref{thm:main}.

Concurrently with this paper, a proof of a weaker version of this conjecture,  with an extra factor of $m^\eps$
in the bound, has appeared in \cite{Sheffer-et-al}.

\bibliographystyle{amsalpha} \bibliography{biblio}

\end{document}